\documentclass[UTF8,10pt]{article}
	
	\usepackage[left=3.18cm,right=3.18cm,top=2.54cm,bottom=2.54cm]{geometry}
	\usepackage{amsfonts}
	\usepackage{amsmath}
	\usepackage{amssymb}
	\usepackage{pgfplots}
	\usepackage{tikz}
	\usetikzlibrary{arrows}
	\usepackage{titlesec}
	\usepackage{bm}
	\usepackage{appendix}
	\usepackage{tikz-cd}
	\usepackage{mathtools}
	\usepackage{amsthm}
	\usepackage{setspace}
	\usepackage{enumitem}
	\setenumerate[1]{itemsep=0pt,partopsep=0pt,parsep=\parskip,topsep=5pt}
	\setitemize[1]{itemsep=0pt,partopsep=0pt,parsep=\parskip,topsep=5pt}
	\setdescription{itemsep=0pt,partopsep=0pt,parsep=\parskip,topsep=5pt}
	\usepackage[colorlinks,linkcolor=black,anchorcolor=black,citecolor=black]{hyperref}
	\usepackage[numbers]{natbib}
	\usepackage{cases}
	\usepackage{fancyhdr}
	\usepackage[scaled=0.92]{helvet}	
	\usepackage{upgreek}
	\usepackage{txfonts}
	\usepackage{verbatim}

	\theoremstyle{definition}
	\newtheorem{thm}{Theorem}[section]

	\newtheorem{lem}[thm]{Lemma}
	\newtheorem{prop}[thm]{Proposition}

	\newtheorem{rmk}{Remark}[section]

	\newcommand{\R}{\mathbb{R}}

	\newcommand{\N}{\mathbb{N}}
	
	\newcommand{\T}{\mathbb{T}}
	\newcommand{\TT}{\mathcal{T}}
	\newcommand{\TP}{\overline{\p}{}}

	\newcommand{\bd}[1]{\mathbf{#1}}  
	\newcommand{\RR}{\mathcal{R}}      
	\newcommand{\Om}{\Omega}
	\newcommand{\q}{\quad}
	\newcommand{\p}{\partial}

	\newcommand{\nab}{\nabla}

	\newcommand{\no}{\nonumber}

	\newcommand{\lleq}{\stackrel{L}{=}}
	\newcommand{\cnab}{\overline{\nab}}
	
	\newcommand{\dx}{\,\mathrm{d}x}

	\newcommand{\lee}{\langle}
	\newcommand{\ree}{\rangle}

	\newcommand{\Er}{\mathring{E}}
	
	\newcommand{\Kr}{\mathring{K}}
	\newcommand{\KKr}{\mathring{\mathcal{K}}}

	\newcommand{\io}{\int_{\Omega}}

	\newcommand{\ddt}{\frac{\mathrm{d}}{\mathrm{d}t}}
	\numberwithin{equation}{section}
	
	\newcommand{\eps}{\varepsilon}

	\usepackage{xcolor}

	\newcommand{\lam}{\lambda}

	\newcommand{\ur}{\mathring{u}}
	\newcommand{\br}{\mathring{B}}
	\newcommand{\pr}{\mathring{p}}
	\newcommand{\sr}{\mathring{S}}
	
	\newcommand{\Dtr}{\mathring{D}_t}
	\newcommand{\rhor}{\mathring{\rho}}
	\newcommand{\nnr}{{(n+1)}}
	\newcommand{\nnn}{{(n)}}
	\newcommand{\nnl}{{(n-1)}}
	\newcommand{\nnll}{{(n-2)}}

	\newcommand{\ff}{\mathcal{F}}
	\newcommand{\ffp}{\mathcal{F}_p}	
	
	\newcommand{\ffpr}{\mathring{\mathcal{F}}_p}

\begin{document}
		\title{\bf Incompressible Limit of Compressible Ideal MHD Flows inside a Perfectly Conducting Wall
		}
		\date{\today}
		\author{{\sc Jiawei WANG}\thanks{Hua Loo-Keng Center for Mathematical Sciences, Academy of Mathematics and Systems Science, Chinese Academy of Sciences, Beijing, P.R.China.
				Email: \texttt{wangjiawei9609@163.com}}
				\,\,\,\,\,\, 
				{\sc Junyan ZHANG}\thanks{Department of Mathematics, National University of Singapore, Singapore. 
				Email: \texttt{zhangjy@nus.edu.sg}} }
		\maketitle
		
		\begin{abstract}
			We prove the incompressible limit of compressible ideal magnetohydrodynamic(MHD) flows in a reference domain where the magnetic field is tangential to the boundary. Unlike the case of transversal magnetic fields, the linearized problem of our case is not well-posed in standard Sobolev space $H^m~(m\geq 2)$, while the incompressible problem is still well-posed in $H^m$. The key observation to overcome the difficulty is a hidden structure contributed by Lorentz force in the vorticity analysis, which reveals that one should trade one normal derivative for two tangential derivatives together with a gain of Mach number weight $\eps^2$. Thus, the energy functional should be defined by using suitable anisotropic Sobolev spaces. The weights of Mach number should be carefully chosen according to the number of tangential derivatives, such that the energy estimates are uniform in Mach number. Besides, part of the proof is similar to the study of compressible water waves, so our result opens the possibility to study the incompressible limit of free-boundary problems in ideal MHD.
			
		\end{abstract}
		
		
		\setcounter{tocdepth}{1}
		\tableofcontents
		
		\section{Introduction}
		
		In this paper, we consider the compressible ideal magnetohydrodynamic(MHD) equations
		\begin{equation}
			\begin{cases}
				D_t\rho+\rho(\nabla\cdot u)=0~~~&\text{in}~[0,T]\times\Om, \\
				\rho D_t u=B\cdot\nabla B-\nabla P,~~P:=p+\frac{1}{2}|B|^2~~~& \text{in}~[0,T]\times\Om, \\
				D_t B=B\cdot\nabla u-B(\nabla\cdot u)~~~&\text{in}~[0,T]\times\Om, \\
				\nabla\cdot B=0~~~&\text{in}~[0,T]\times\Om,\\
				D_tS=0~~~&\text{in}~[0,T]\times\Om,
			\end{cases}\label{CMHD}
		\end{equation}
		describing the motion of a compressible conducting fluid in an electro-magnetic field. Here $\Om:=\mathbb{T}^{d-1}\times (-1,1)$ is the reference domain in $\R^d~(d=2,3)$ with boundary $\Sigma:=\Sigma_+\cup\Sigma_-$ and $\Sigma_\pm:=\{x_d=\pm 1\}$. $\nabla:=(\p_{x_1},\cdots,\p_{x_d})^\mathrm{T}$ is the standard spatial derivative. $D_t:=\p_t+u\cdot\nabla$ is the material derivative. The fluid velocity, the magnetic field, the fluid density, the fluid pressure and the entropy are denoted by $u=(u_1,\cdots,u_d)^\mathrm{T}$, $B=(B_1,\cdots,B_d)^\mathrm{T}$, $\rho$, $p$ and $S$ respectively. Note that the last equation of \eqref{CMHD} is derived from the equation of total energy and Gibbs relation. See \cite[Ch. 4.3]{MHDphy} for more details. Throughout this paper, we will use follow the convention of Einstein's summation, that is, repeated indices represent taking sum over the indices.

We assume the fluid pressure $p=p(\rho, S)$ to be a given smooth function of $\rho, S$ which satisfies 
		\begin{align}
			\rho>0,~~~\frac{\p p}{\p\rho}> 0, \q \text{in}\,\, \bar{\Om}.
			\label{EoS}
		\end{align}
		These two conditions also guarantee the hyperbolicity of system \eqref{CMHD}.

		The initial and boundary conditions of system \eqref{CMHD} are
		\begin{align}
	\label{data}	    (u,B,\rho,S)|_{t=0}=(u_0, B_0, \rho_0,S_0) ~~~& \text{in }[0,T]\times\Om,\\
  \label{bdry cond}			u_d=B_d=0 ~~~& \text{on }[0,T]\times\Sigma,
		\end{align} where the boundary condition for $u_d$ is the slip boundary condition, and the boundary condition for $B_d$ shows that $\Sigma_\pm$ are perfectly conducting walls.

	    \begin{rmk} The conditions $\nabla\cdot B=0$ in $\Om$ and $B_d=0$ on $\Sigma$ are both constraints for initial data so that the MHD system is not over-determined. One can show that they propagate within the lifespan of the solution. Using the theory of hyperbolic system with charateristic boundary condition \cite{Rauch1985}, one can show that the correct number of boundary conditions when assuming $u_d|_{\Sigma}=0$ is 1 (see \cite{TW2020MHDLWP}). So, $B_d|_{\Sigma}=0$ has to be an initial constraint.
	    \end{rmk}

The initial-boundary value problem \eqref{CMHD}-\eqref{bdry cond} is used to characterize the motion of a plasma confined in a rigid perfectly conducting wall, which is an important model in the study of laboratory plasma confinement problems. See \cite[Ch. 4.6]{MHDphy} for more details. To make the initial-boundary value problem \eqref{CMHD}-\eqref{bdry cond} solvable, we need to require the initial data satisfying the compatibility conditions up to certain order. For $m\in\N$, we define the $m$-th order compatibility conditions to be
\begin{equation}\label{comp cond}
B_d|_{\{t=0\}\times\Sigma}=0,~~\p_t^j u_d|_{\{t=0\}\times\Sigma}=0,~~0\leq j\leq m.
\end{equation}It should be noted that $ \p_t^j B_d|_{\{t=0\}\times\Sigma}=0$ also holds for $1\leq j\leq m$ provided \eqref{comp cond} holds.

		\subsection{The equation of state and sound speed} \label{sect mach number definition}

		This paper is devoted to studying the behavior of the solution of \eqref{CMHD} as the sound speed goes to infinity, which is known to be the incompressible limit. Physically, the sound speed is defined by $c_s:=\sqrt{\frac{\p p}{\p\rho}}$. Mathematically, it is convenient to view the sound speed as a parameter $\lambda$. A typical choice of the equation of states $p_{\lam}(\rho, S)$ would be the polytropic gas
		\begin{align}
			p_{\lam}(\rho,S)=\lam^2\left(\rho^{\gamma}\exp(S/C_V)-1\right),~~~\gamma\geq 1,~~C_V>0.
		\end{align}
		 When viewing the density as a function of the pressure and the entropy, this indicates
		\begin{align}\label{eos2}
			\rho_{\lam} (p,S) =  \left(\left(1+\frac{p}{\lam^2}\right)e^{-\frac{S}{C_V}}\right)^{\frac{1}{\gamma}}, \quad \text{and}\,\,\log \left(\rho_{\lam} (p,S)\right) = \gamma^{-1} \log\left(\left(1+\frac{p}{\lam^2}\right)e^{-\frac{S}{C_V}}\right). 
		\end{align} Let $\ff:= \log \rho$. Since $\frac{\p p}{\p\rho}>0$ indicates $\frac{\p\ff}{\p p}>0$, using $D_t S=0$, the first equation of \eqref{CMHD} is equivalent to 
		\begin{equation} \label{continuity eq f}
			\frac{\p\ff}{\p p} D_t p  +\nab\cdot u=0. 
		\end{equation}
		Hence,  we can view $\ff=\log\rho$ as a parametrized family $\{\ff_\eps (p,S)\}$ as well, where $\eps = \frac{1}{\lam}$. 
		Since we work on the case when the entropy and velocity are both bounded (later we will prove $\|u,S,\eps p\|_{H^4(\Om)}$ is bounded uniformly in $\eps$), we again slightly abuse the terminology, and call $\lam$ the sound speed and call $\eps$ the Mach number and thus $M=O(\eps)$. It is easy to see that there exists $C>0$ such that
		\begin{align}
			\frac{1}{\rho}\frac{\p\rho}{\p p}(p,S)\leq C \eps^2. \label{ff'}
		\end{align} Furthermore, if $\eps>0$ is sufficiently small, there exists $C>0$ such that
		\begin{align}
			C^{-1} \eps^2 \leq \frac{\p\ff_\eps}{\p p} (p,S)=\frac{1}{\rho}\frac{\p\rho}{\p p}(p,S)\leq C \eps^2. \label{ff' is lambda}
		\end{align}
		Also, we have for some $C_k>0$
		\begin{equation}
			\frac{\p\ff_\eps}{\p p} = \gamma^{-1} \frac{\eps^2}{1+\eps^2 p},~~\frac{\p^k\ff_\eps}{\p p^k}=C_k\frac{\eps^{2k}}{(1+\eps^2 p)^{k}}, ~~~k\in\N^*,
		\end{equation} and so 
		\begin{align} \label{ff property}
			|\p_p^k\ff_\eps(p,S) |\leq C, \quad |\p_p^k\ff_\eps(p,S)| \leq  C \p_p\ff_{\eps} (p,S)~~0\leq k\leq 8,~k\in\N.
		\end{align}
		 Writing $\ffp:=\frac{\p \ff}{\p p}$, system \eqref{CMHD} is reformulated as follows
		\begin{equation}\label{CMHD3}
			\begin{cases}
				\ffp D_t p  +\nab\cdot u=0~~~&\text{in}~[0,T]\times\Om, \\
				\rho D_t u=B\cdot\nabla B-\nabla P,~~P:=p+\frac{1}{2}|B|^2~~~& \text{in}~[0,T]\times\Om, \\
				D_t B=B\cdot\nabla u-B(\nabla\cdot u)~~~&\text{in}~[0,T]\times\Om, \\
				\nabla\cdot B=0~~~&\text{in}~[0,T]\times\Om,\\
				D_t S=0~~~&\text{in}~[0,T]\times\Om,\\
	                          p=p(\rho,S),\frac{\p p}{\p\rho}>0~~~&\text{in}~[0,T]\times\bar{\Om},\\
				u_d=B_d=0 ~~~& \text{on}~[0,T]\times\Sigma,\\
				(u,B,\rho,S)|_{t=0}=(u_0, B_0, \rho_0,S_0)~~~& \text{on}~\{t=0\}\times\Om.\\
			\end{cases}
		\end{equation}
		When considering the incompressible limit, we sometimes replace $\ffp$ by $\eps^2$ for simplicity of notations.
		
		\subsection{An overview of previous results}\label{sect history}
		
The incompressible limit of compressible inviscid fluids is considered to be a type of singular limit of hyperbolic system: the pressure for compressible fluids is a variable of hyperbolic system whereas the pressure for incompressible fluids is a Lagrangian multiplier and the equation of state is no longer valid. Early works about compressible Euler equations can be dated back to Klainerman-Majda \cite{Klainerman1981limit,Klainerman1982limit} when the domain is the whole space $\R^d$ or the periodic domain $\T^d$,  Ebin \cite{Ebin1982limit} and Schochet \cite{Schochet1986limit, Schochet1987limit} when the domain is bounded, and Isozaki \cite{Isozaki1987limit} when considering an exterior domain. The abovementioned papers consider the case of ``well-prepared" initial data, which means the compressible initial data is exactly a small perturbation of a given incompressible initial data. When the initial data is ``ill-prepared", that is, the compressible initial data is the small perturbation of incompressible initial data plus a highly oscillatory part, we refer to \cite{Ukai1986limit, Asano1987limit, Isozaki1987limit, Schochet1994limit, Iguchi1997limit, Secchi2000, Metivier2001limit, Alazard2005limit} and references therein. The precise definitions of ``well-prepared data" and ``ill-prepared data" can be found in \cite{Secchi2000, Metivier2001limit, Alazard2005limit}.

For compressible ideal MHD, the incompressible limit in the whole space was studied by Jiang-Ju-Li \cite{JJLMHDlimit2}. However, when the domain has a boundary, the study of incompressible limit for ideal compressible MHD is much more subtle. The first difficulty is that the vorticity estimate for ideal compressible MHD cannot be closed when using standard Sobolev spaces, and thus the method presented in \cite{Schochet1986limit, Schochet1987limit} is no longer valid. When the magnetic field is not tangential to the boundary, one can use such transversality to compensate the loss of normal derivative arising in the vorticity analysis. See Yanagisawa \cite{1987MHDfirst} for the well-posedness result under the condition $B\times N|_{\Sigma}=\mathbf{0}$. As for the singular limits, we refer to Ju-Schochet-Xu \cite{JuMHD2} for the singular limits when both Mach number and Alfv\'en number converge to zero and Jiang-Ju-Xu \cite{JuMHD1} for the low Alfv\'en number limit of incompressible ideal MHD whose local existence was proved via a low-Mach-number regime.

Unfortunately, when the magnetic field is tangential to the boundary, that is, $B\cdot N|_{\Sigma}=0$, Ohno-Shirota \cite{OS1998MHDill} proved that the linearized problem near a non-zero magnetic field is ill-posed in standard Sobolev spaces. In this case, one has to use suitable anisotropic Sobolev spaces introduced by Chen \cite{ChenSX} to proved the well-posedness. See Yanagisawa-Matsumura \cite{1991MHDfirst} and Secchi \cite{Secchi1995, Secchi1996}.

Therefore, an essential difficulty in establishing the incompressible limit for ideal MHD with $B\cdot N|_{\Sigma}=0$ is the ``incompatibility of function spaces for the existence": compressible ideal MHD flow with the perfectly conducting wall condition may not be well-posed in standard Sobolev spaces, while the corresponding incompressible problem is well-posed (see, for example, Gu-Wang \cite{GuWang2016LWP}). So far, there is only a partial answer for this incompressible limit problem. Ju and the first author \cite{JuWangMHDlimit} proved the incompressible limit in a suitable closed subspace of standard Sobolev space, introduced by Secchi \cite{Secchi1995-2}, by adding more restrictive constraints for the boundary value of initial data. Specifically, \cite{Secchi1995-2, JuWangMHDlimit} require the initial data to satisfy
\[
\p_3^{2k}(u_3,B_3)|_{\Sigma}=0,~\p_3^{2k+1}(p,u_1,u_2,B_1,B_2,S)|_{\Sigma}=0,~~k=0,1,\cdots,\lfloor\frac{m-1}{2}\rfloor.
\] However, the physical interpretation of these ``extra constraints" is still unclear. In other words, it is still unknown how to thoroughly overcome the difficulty caused by the abovementioned ``incompatibility". 

The aim of this paper is to give a definitive answer to the incompressible limit problem of ideal MHD with the perfectly conducting wall condition when the initial data is ``well-prepared". Indeed, in the vorticity analysis, there is a hidden structure brought by the Lorentz force term. To the best of our knowledge, such an observation never appears in any previous works, but it can really help us easily prove the uniform estimates in Mach number and also illustrates why the anisotropic Sobolev spaces are naturally introduced and necessary for compressible MHD but are not needed for incompressible MHD.

		\subsection{The main theorems}		
		Before stating our results, we should first define the anisotropic Sobolev space $H_*^m(\Omega)$ for $m\in\N$ and $\Om=\T^{d-1}\times[-1,1]$. Let $\omega=\omega(x_d)$ be a cutoff function\footnote{The choice of $\omega(x_d)$ is not unique. We just need $\omega(x_d)$ vanishes on $\Sigma$ and is comparable to the distance function near $\Sigma$.} on $[-1,1]$ defined by $\omega(x_d)=(1-x_d)(1+x_d)$. Then we define $H_*^m(\Omega)$ for $m\in\N^*$ as follows
\[
H_*^m(\Omega):=\left\{f\in L^2(\Omega)\bigg| (\omega\p_d)^{\alpha_{d+1}}\p_1^{\alpha_1}\cdots\p_d^{\alpha_d} f\in L^2(\Omega),~~\forall \alpha \text{ with } \sum_{j=1}^{d-1}\alpha_j +2\alpha_d+\alpha_{d+1}\leq m\right\},
\]equipped with the norm
\begin{equation}\label{anisotropic1}
\|f\|_{H_*^m(\Omega)}^2:=\sum_{\sum_{j=1}^{d-1}\alpha_j +2\alpha_d+\alpha_{d+1}\leq m}\|(\omega\p_d)^{\alpha_{d+1}}\p_1^{\alpha_1}\cdots\p_d^{\alpha_d} f\|_{L^2(\Omega)}^2.
\end{equation} For any multi-index $\alpha:=(\alpha_0,\alpha_1,\cdots,\alpha_{d},\alpha_{d+1})\in\N^{d+2}$, we define
\[
\p_*^\alpha:=\p_t^{\alpha_0}(\omega\p_3)^{\alpha_{d+1}}\p_1^{\alpha_1}\cdots\p_d^{\alpha_d},~~\lee \alpha\ree:=\sum_{j=0}^{d-1}\alpha_j +2\alpha_d+\alpha_{d+1},
\]and define the \textbf{space-time anisotropic Sobolev norm} $\|\cdot\|_{m,*}$ to be
\begin{equation}\label{anisotropic2}
\|f\|_{m,*}^2:=\sum_{\lee\alpha\ree\leq m}\|\p_*^\alpha f\|_{L^2(\Omega)}^2=\sum_{\alpha_0\leq m}\|\p_t^{\alpha_0}f\|_{H_*^{m-\alpha_0}(\Omega)}^2.
\end{equation}

We also denote the interior Sobolev norm to be $\|f\|_{s}:= \|f(t,\cdot)\|_{H^s(\Omega)}$ for any function $f(t,x)\text{ on }[0,T]\times\Omega$ and denote the boundary Sobolev norm to be $|f|_{s}:= |f(t,\cdot)|_{H^s(\Sigma)}$ for any function $f(t,x)\text{ on }[0,T]\times\Sigma$. From now on, we assume the dimension to be $d=3$, that is, $\Om=\T^2\times(-1,1)$ and $\Sigma_\pm=\{x_3=\pm 1\}$. In the proof, we will see the 2D case follow in the same manner as the 3D case up to a slight modification in the vorticity analysis. 

First, we prove the local existence togther with uniform-in-$\eps$ estimates. 
		\begin{thm}[Local well-posedness and uniform-in-$\eps$ estimates]\label{main thm, well data}
			Let $\eps\in(0,1)$ be fixed. Let $(u_0, B_0, \rho_0,S_0)\in H^8(\Om)\times H^8(\Om)\times H^8(\Om)\times H^8(\Om)$ be the initial data of \eqref{CMHD3} satisfying the compatibility conditions \eqref{comp cond} up to 7-th order and
			\begin{equation}
				E(0)\le M
			\end{equation}
		    for some $M>0$ independent of $\eps$. Then there exists $T>0$ depending only on $M$, such that \eqref{CMHD3} admits a unique solution $(p(t),u(t),B(t),S(t))$ that verifies the energy estimate
		    \begin{equation}
		    	\sup_{t\in[0,T]}E(t) \le P(E(0)),
		    \end{equation}
	        where $P(\cdots)$ is a generic polynomial in its arguments, and the energy $E(t)$ is defined to be
	        \begin{equation}
	        		\begin{aligned}\label{energy intro}
	        			E(t)=&~E_4(t) + E_5(t) + E_6(t) + E_7(t) + E_8(t),\\
	        			E_{4+l}(t)=&\sum_{\lee\alpha\ree=2l}\sum_{k=0}^{4-l}\left\|\eps^{(k-1)_++2l}  \TT^{\alpha}\p_t^k\left(u,B,S,\ffp^{\frac{(k+\alpha_0-l-3)_+}{2}}p\right)\right\|_{4-k-l}^2~~0\leq l\leq 4,
	        		\end{aligned}
	        \end{equation}
            where $K_+:=\max\{K,0\}$ and we denote $\TT^{\alpha}:=(\omega(x_3)\p_3)^{\alpha_4}\p_t^{\alpha_0}\p_1^{\alpha_1}\p_2^{\alpha_2}$ to be a high-order tangential derivative for the multi-index $\alpha=(\alpha_0,\alpha_1,\alpha_2,0,\alpha_4)$ with length (for the anisotropic Sobolev spaces) $\lee \alpha\ree=\alpha_0+\alpha_1+\alpha_2+2\times0+\alpha_4$. In the rest of this paper, we sometimes write $\TT^k$ to represent a tangential derivative $\TT^{\alpha}$ with order $\lee\alpha\ree =k$ when we do not need to specify what the derivative $\TT^{\alpha}$ contains.
		\end{thm}
		   
		\begin{rmk}[Correction of $E_4(t)$]
		We note that the norm $\|p\|_4^2$ in $E_4(t)$ defined by \eqref{energy intro} should be replaced by $\|\sqrt{\ffp} p\|_0^2+\|\nab p\|_3^2$ because we do not have $L^2$ estimates of $p$ without $\eps$ weight. We still write $\|p\|_4^2$ as above for simplicity of notations.
		\end{rmk}

		\begin{rmk}[``Prepared" initial data]
			The above estimate only requires $\nab\cdot u_0=O(\eps)$ and $\p_t u|_{t=0}=O(1)$. In this case, the compressible data $u_0$ is a small perturbation of an incompressible data $u_0^0$, and this perturbation is \textit{completely contributed by the compressibility}. Such compressible data are usually called ``well-prepared initial data"\footnote{One can find the definitions of ``well-prepared" and ``ill-prepared" in \cite{Metivier2001limit, Alazard2005limit} for rescaled Euler system, which is equivalent to the statement in our paper.}.
		\end{rmk}

			\begin{rmk}[Weights of Mach number of $p$]
		In \eqref{energy intro}, the weight of Mach number of $p$ is slightly different from $(u,B,S)$, but such difference only occurs when $\TT^\alpha$ are full time derivatives and $k=4-l$. In fact, due to $k\leq 4-l$ and $\alpha_0\leq \lee\alpha\ree=2l$, we know $(k+\alpha_0-l-3)_+$ is always equal to zero unless $\alpha_0=2l$ and $k=4-l$ simultanously hold.
		\end{rmk}

	\begin{rmk}[Relations with anisotropic Sobolev space]
		The energy functional $E(t)$ above is considered as a variant of $\|\cdot\|_{8,*}$ norm at time $t>0$. For different multi-index $\alpha$, we impose different Mach number weights according to the number of tangential derivatives that appear in $\p_*^{\alpha}$, such that the energy estimate for the slightly modified $\|\cdot\|_{8,*}$ norm is uniform in $\eps>0$.
		\end{rmk}	

		\begin{rmk}[Choice of regularity]
		We propose $H^4$ regularity for the energy functional because lots of commutator estimates require the bound for $\|\TP^2(u,B)\|_{L^{\infty}}$. Recall that $H^{\frac{d}{2}+\delta}\hookrightarrow L^{\infty}$, so it would be convenient to choose $H^{2+\lceil\frac{d}{2}+\delta\rceil}$ regularity, that is, $H^4$ for $d=2,3$. It should be noted that previous results about local existence of \eqref{CMHD3}, e.g, \cite{Secchi1995, Secchi1996, Shizuta1994}, require the initial data to belong to $H_*^{2m}(\Om)$ for $m\geq 8$ instead of $m\geq 4$. In Appendix \ref{apdx LWP}, we establish the local well-posedness of \eqref{CMHD3} for initial data belonging to $H_*^8(\Om)$ and satisfying the assumption of Theorem \ref{main thm, well data}. The energy functional used to prove the local existence is still $E(t)$. With this new existence theorem and following the same strategies, the results in this paper can be easily generalized to the case when $H^4$ and $H_*^8$ are replaced by $H^m$ and $H_*^{2m}~(m\geq 4)$ respectively.
		\end{rmk}	

		\begin{rmk}[Compatibility conditions and regularity of initial data]
		The initial data are assumed to be $H^8(\Om)$ functions instead of merely $H_*^8(\Om)$ functions because we have to guarantee they satisfy the compatibility conditions up to 7-th order and the uniform-in-$\eps$ bound $E(0)\leq M$. For example, the 7-th order compatibility conditions requires $\p_t^7u_3|_{\{t=0\}\times\Sigma}=0$, but $\p_t^7 u|_{t=0}$ only belongs to $H_*^1(\Om)$ whose trace on the boundary may have no meaning if we choose initial data from $H_*^8(\Om)$. On the other hand, if we alternatively choose initial data in $H^8(\Om)$, then at least we have $\p_t^7u|_{t=0}\in H^1(\Om)$ whose trace on $\Sigma$ can be defined in $H^{\frac12}(\Sigma)$. See also \cite[Theorem 2.1]{Secchi1996}, \cite[Section 8.3]{Zhang2021CMHD} for the discussion on this issue.
		\end{rmk}

		The next main theorem concerns the incompressible limit. We consider the incompressible MHD equations together with a transport equation satisfied by $(u^0,B^0,\pi,S^0)$ with incompressible initial data $(u_0^0, B_0^0)$ and $S_0^0$:
		\begin{equation} \label{IMHD}
			\begin{cases}
				\varrho(\p_t u^0 + u^0\cdot\nab u^0) -B^0\cdot\nab B^0+ \nab (\pi+\frac12|B^0|^2) =0&~~~ \text{in}~[0,T]\times \Omega,\\
				\p_t B^0+ u^0\cdot\nab B^0 -B^0\cdot\nab u^0=0 &~~~ \text{in}~[0,T]\times \Omega,\\
				\p_tS^0+u^0\cdot\nab S^0=0&~~~ \text{in}~[0,T]\times \Omega,\\
				\nab\cdot u^0=\nab\cdot B^0=0&~~~ \text{in}~[0,T]\times \Omega,\\
				u_3^0=B_3^0=0&~~~\text{on}~[0,T]\times\Sigma,\\
				(u_0,B_0,S_0)|_{t=0}=(u_0^0, B_0^0,S_0^0)&~~~\text{on}~\{t=0\}\times\Om.
			\end{cases}
		\end{equation}

		\begin{thm}[\textbf{Incompressible limit}] \label{main thm, limit}
			Under the hypothesis of Theorem \ref{main thm, well data}, we assume that $(u_0,B_0,S_0) \to (u_0^0,B_0^0,S_0^0)$ in $H^4(\Omega)$ as $\eps\to 0$ for $\nab\cdot u_0^0=\nab\cdot B_0^0=0$ in $\Om$ with $u_{03}^0=B_{03}^0=0$ on $\Sigma$. Then it holds that 
			\begin{align*}
				(u,B,S) \to& (u^0,B^0,S^0) \quad \text{weakly-* in } L^{\infty}([0,T];H^4(\Om)) \text{ and strongly in } C([0,T];H^{4-\delta}(\Om))
			\end{align*}
			for $\delta>0$. $(u^0,B^0,S^0)$ solves \eqref{IMHD}, that is, the incompressible MHD equations together with a transport equation of $S^0$. Here $\varrho$ satisfies the transport equation $$\p_t\varrho+u^0\cdot\nab \varrho =0,~~\varrho|_{t=0}=\rho(0,S_0^0),$$ where we consider the equation of state \eqref{eos2} as $\rho=\rho(\eps^2 p, S)$, that is, a function of $\eps^2 p$ and $S$.Also, there exists a function $\pi$ such that $\nab\pi\in C([0,T];H^3(\Om))$ satisfies the incompressible MHD system \eqref{IMHD} and the convergence result
\[
\rho^{-1}\nab p\to \varrho^{-1}\nab \pi\quad \text{weakly-* in } L^{\infty}([0,T];H^3(\Om)) \text{ and strongly in } C([0,T];H^{3-\delta}(\Om)).
\]
		\end{thm}

		\begin{rmk}[The space for the convergence of initial data]
		The compressible initial data converges to the incompressible data in $H^4(\Om)$ instead of $H_*^8(\Om)$ because the higher-order energy $E_5(0)\sim E_8(0)$ automatically vanishes as $\eps\to 0$. Thus, it suffices to require the convergence in $H^4(\Om)$.
		\end{rmk}	

		\begin{rmk}[Boundedness of the domain]
		In the proof of the main theorem, we do not use the boundedness of the domain, such as Poincar\'e's inequality, to close the energy estimate. Thus, after making necessary modifications (including setting $\rho_0-1\in H_*^8(\Om)$ instead of $\rho_0$ itself and assuming the initial data to be ``localized"), our proof is still valid for the case of unbounded reference domains such as $\R^{d-1}\times(-1,1)$ and the half space. Indeed, using a partition of unity as in \cite{CS2007LWP}, the proof can be directly generalized to the domains which are diffeomorphic to the reference domains via $H_*^8(\Om)$-diffeomorphisms.
		\end{rmk}

			\subsection{Organization of the paper}\label{sect org} 
This paper is organized as follows. In section \ref{sect overview intro}, we discuss the main difficulties and briefly introduce our strategies to tackle the problem. Then section \ref{sect uniform} is devoted to the proof of uniform estimates in Mach number. Combined with the compactness argument, we conclude the incompressible limit in section \ref{sect limit}. In Appendix \ref{apdx LWP}, we prove the local existence of \eqref{CMHD3} for initial data in $H_*^8(\Om)$.

		\paragraph*{Acknowledgment.} The research of Jiawei Wang is supported by the National Natural Science Foundation of China (Grants 12131007 and 12071044) and the Basic Science Center Program (No: 12288201) of the National Natural Science Foundation of China. Jiawei Wang would like to thank his Ph.D. advisor Prof. Qiangchang Ju for helpful suggestions. Junyan Zhang would like to thank Chenyun Luo for the kind hospitality during his visit at The Chinese University of Hong Kong.

		\subsection*{List of Notations}
		\begin{itemize}
			\item (Tangential derivatives) $\TT_0= \p_t$ denotes the time derivative, $\TT_j=\TP_j~(1\leq j\leq d-1)$ denotes the tangential spatial derivatives and $\TT_{d+1}:=\omega(x_d)\p_d$ with $\omega(x_d)=(1-x_d)(1+x_d)$.
			\item (Tangential components) $\bar{u}:=(u_1,\cdots,u_{d-1})^\mathrm{T}$ and $\bar{B}:=(B_1,\cdots,B_{d-1})^\mathrm{T}$. We also write $\cnab:=(\TP_1,\cdots,\TP_{d-1})^\mathrm{T}$ as the tangential components of the operator $\nab$.
			\item ($L^\infty$-norm) $\|\cdot\|_{\infty}:= \|\cdot\|_{L^\infty(\Omega)}$, {$|\cdot|_{\infty}:=\|\cdot\|_{L^\infty(\Sigma)}$}. 
			\item (Interior Sobolev norm) $\|\cdot\|_{s}$:  We denote $\|f\|_{s}:= \|f(t,\cdot)\|_{H^s(\Omega)}$ for any function $f(t,y)\text{ on }[0,T]\times\Omega$.
			\item  (Boundary Sobolev norm) $|\cdot|_{s}$:  We denote $|f|_{s}:= |f(t,\cdot)|_{H^s(\Sigma)}$ for any function $f(t,y)\text{ on }[0,T]\times\Sigma$.
			\item (Anisotropic Sobolev norms) $\|\cdot\|_{m,*}$: For any function $f(t,x)\text{ on }[0,T]\times\Omega$, $\|f\|_{m,*}^2:= \sum_{\lee\alpha \ree\leq m}\|\p_*^\alpha f(t,\cdot)\|_{0}^2$ denotes the $m$-th order space-time anisotropic Sobolev norm of $f$.
			\item (Polynomials) $P(\cdots)$ denotes a generic polynomial in its arguments.
			\item (Commutators) $[T,f]g=T(fg)-f(Tg)$, $[f,T]g=-[T,f]g$, $[T,f,g]:=T(fg)-T(f)g-fT(g)$ where $T$ is a differential operator and $f,g$ are functions.
			\item (Equality modulo lower order terms) $A\lleq B$ means $A=B$ modulo lower order terms. 
		\end{itemize}

		\section{Difficulties and strategies}\label{sect overview intro}
	
		Before going to the detailed proofs, we will discuss the main difficulties in this incompressible limit problem and briefly introduce our strategies to derive energy estimates that are uniform in the Mach number.  Note that our initial data is assumed to be well-prepared, so the uniform estimates together with compactness argument are enough to derive the incompressible limit.

		\subsection{Choice of function spaces}\label{sect stat 1}
		Compressible ideal MHD with perfectly conducting wall conditions is a first-order quasilinear hyperbolic system with \textit{characteristic} boundary conditions of constant multiplicity, for which there is a potential of normal derivative loss. In fact, inside a rigid wall with the slip boundary condition, Euler equations and elastodynamic equations with degenerate deformation tensor on the boundary (see \cite{Zhang2021elasto}) also have characteristic boundary conditions, but their vorticity can be controlled in the setting of standard Sobolev spaces because there are no other quantities involved in the pressure part. The loss of normal derivatives is then compensated via the control of vorticity and divergence. However, the vorticity estimates for compressible ideal MHD cannot be closed in the setting of standard Sobolev spaces, which was explicitly presented in \cite{Zhang2020CRMHD}. 

		On the one hand, people observed that such normal derivative loss can be compensated if the vertical component of the magnetic field does not vanish on the boundary. See Yanagisawa \cite{1987MHDfirst} for the well-posedness and \cite{MTT2018MHDCD,TW2021MHDCDST,WangXinMHDCD} for the study of corresponding free-boundary problems (MHD contact discontinuities). But the transversality of magnetic fields violates the perfectly conducting wall condition. On the other hand, Chen \cite{ChenSX} first introduced the anisotropic Sobolev spaces defined in \eqref{anisotropic1}-\eqref{anisotropic2} in the study of compressible gas dynamics inside a rigid wall. Under this setting, a normal derivative is considered as a second-order derivative, which exactly compensates the derivative loss. Detailed analysis for MHD in such setting is referred to the second author's paper \cite[Sect. 2.5]{Zhang2021CMHD}. Using such function spaces, Yanagisawa-Matsumura \cite{1991MHDfirst} and Secchi \cite{Secchi1995, Secchi1996} proved the local well-posedness and see also \cite{Trakhinin2008CMHDVS,TW2020MHDLWP,Zhang2021CMHD} for the study of free-boundary problems, but the estimates obtained in these previous works are not uniform in Mach number.

		As stated in Section \ref{sect history}, the essential difficulty in proving the incompressible limit is that the function spaces for well-posedness of compressible problem are not ``compatible" with the ones for the incompressible problems. Since our initial data is well-prepared, that is, a given incompressible data plus a slight perturbation, we shall still start with the incompressible counterpart and try to find out the relationships between the incompressible problem and the compressible problem.

		\subsection{Key observation: hidden structure of Lorentz force in the vorticity analysis}\label{sect stat 2}

		First, the entropy is easy to control thanks to $D_t S=0$, so it suffices to analyze the relations between $(u,B)$ and $p$. We start with the control of $\|(u,B)\|_4$. Using div-curl decomposition, the $H^4$ Sobolev norms are bounded by $\|\nab\times u,\nab\times B\|_3$, $\|\nab\cdot u,\nab\cdot B\|_3$ and the normal traces $|u_3,B_3|_{3.5}$. The boundary conditions \eqref{bdry cond} eliminate the normal traces, and the divergence part is reduced to tangential derivatives $\|\eps^2 D_t p\|_3$  (here we temporarily abuse the notations to write $\ffp=\eps^2$) thanks to the continuity equation and the divergence constraint $\nab\cdot B=0$. As for the vorticity, taking curl in the momentum equation yields the evolution equation of vorticity
\begin{equation}
\rho D_t(\nab\times u) - (B\cdot\nab)(\nab\times B) = \cdots,
\end{equation}where we find that the term $\nab (p+\frac12|B|^2)$ is already eliminated. In the $H^3$-control of $\nab\times u$, we invoke the evolution equation of $B$ to get
\begin{equation}\label{bad1}
\ddt\io \rho|\p^3(\nab\times u)|^2+|\p^3(\nab\times B)|^2\dx= -\io \p^3\nab\times(B(\nab\cdot u))\cdot\p^3(\nab\times B)\dx+ \text{ controllable terms},
\end{equation}where we find that there are 5 derivatives falling on $u$ and thus the vorticity estimates cannot be closed in the setting of 4-th order standard Sobolev spaces. Such derivative loss in the vorticity analysis must appear because taking curl operator eliminates the term $\nab (\frac12|B|^2)$ in the momentum equation but does not eliminate the term $B(\nab\cdot u)$. However, in the proof of $L^2$ estimate, the contribution of $\nab (\frac12|B|^2)$ should be cancelled with the contribution of $B(\nab\cdot u)$.

If we further analyze this problematic term by using the structure of MHD system, we will find that the anisotropic Sobolev space is naturally introduced in order to close the energy estimates. Using the continuity equation $\eps^2 D_t p=-\nab\cdot u$, we find that the highest order term in \eqref{bad1} is $\eps^2 B\times (\p^3\nab D_tp)$. Then commuting $\nab$ with $D_t$ and invoking the momentum equation $\rho D_tu+B\times(\nab\times B)=-\nab p$, we get
\begin{equation}\label{key}
\eps^2 B\times (\p^3\nab D_tp)\lleq\eps^2 B\times (\p^3 D_t\nab p)\lleq -\eps^2\rho B\times(\p^3D_t^2 u)-\eps^2 B\times  D_t(B\times\p^3(\nab\times B)).
\end{equation} On the right side of the above identity, the second term still contains 4 normal derivatives $\p^3 \nab\times$ and 1 tangential derivative $D_t$ which exhibits a loss of derivative. The key observation here is that the contribution of this term in \eqref{bad1} gives an energy term and thus avoid the loss of derivative. Specifically, using $\nab\times(B(\nab\cdot u))=(\nab\times B)(\nab\cdot u)-B\times\nab(\nab\cdot u)$ and omitting the contribution of $(\nab\times B)(\nab\cdot u)$ in \eqref{bad1} (this part has no loss of derivative), we have
\begin{equation}\label{bad0}
\begin{aligned}
&-\io \p^3\nab\times(B(\nab\cdot u))\cdot\p^3(\nab\times B)\dx\lleq \io B\times\p^3\nab\underbrace{(\nab\cdot u)}_{=-\eps^2 D_t p}\cdot\p^3(\nab\times B)\dx\\
\lleq&\io\eps^2 \left(\rho B\times(\p^3D_t^2 u)\right)\cdot\p^3(\nab\times B)\dx+\io \eps^2 \left(B\times D_t\left(B\times\p^3 (\nab\times B)\right)\right)\cdot\p^3(\nab\times B)\dx,
\end{aligned}
\end{equation}where we insert \eqref{key} in the second equality. Then setting $\bd{u}=B$, $\bd{v}=D_t\left(B\times\p^3 (\nab\times B)\right)$ and $\bd{w}=\p^3(\nab\times B)$ in the vector identity $(\bd{u}\times \bd{v})\cdot\bd{w}=-(\bd{u}\times \bd{w})\cdot\bd{v}$, we find that the last term in \eqref{bad0} contributes to an energy term
\begin{equation}\label{bad3}
\begin{aligned}
&\io \eps^2 \left(B\times D_t\left(B\times\p^3(\nab\times B)\right)\right)\cdot\p^3(\nab\times B)\dx\\
=&-\io\eps^2\left(B\times\p^3(\nab\times B)\right)\cdot D_t\left(B\times\p^3 (\nab\times B)\right)\dx\\
=&-\frac12\ddt\io\eps^2\left|B\times\p^3(\nab\times B)\right|^2\dx+ \text{ controllable terms}.
\end{aligned}
\end{equation}

The above analysis for compressible ideal MHD shows that the vorticity estimates cannot be closed in standard Sobolev spaces, but in fact we \textit{trade one normal derivative} (in $\nab\times$) \textit{for two tangential derivatives} $\eps^2D_t^2$ thanks to the special structure of Lorentz force $B\times(\nab\times B)$ whose contribution in the vorticity analysis is an $\eps$-weighted energy term instead of introducing loss of derivatives. This exactly illustrates why the anisotropic Sobolev space is naturally introduced to study compressible ideal MHD. 

It should also be noted that this difficulty never occurs for incompressible ideal MHD because the divergence-free condition automatically eliminates the bad term above. Indeed, taking the limit $\eps\to 0_+$, the right side of \eqref{bad1} automatically vanishes and the vorticity estimates for the incompressible problem can be established in standard Sobolev spaces.  Besides, to prove the incompressible limit of ideal MHD with \eqref{bdry cond}, an extra Mach number weight $\eps^2$ must be added together with the two tangential derivatives. So, we find that the ``anisotropic part" of the energy, namely the terms $E_5\sim E_8$ in \eqref{energy intro}, will automatically vanish when we take the incompressible limit $\eps\to 0_+$. The remaining term $E_4(t)$ exactly consists of the standard Sobolev norms, which coincides with the energy for incompressible ideal MHD.

\begin{rmk}[Slight modifications in 2D]
When dimension $d=2$, we have to replace the curl operator $\nab\times$ by $\nab^\perp\cdot$ where $\nab^\perp:=(-\p_2,\p_1)$. Then \eqref{bad1} reads 
\begin{equation}\label{bad2}
\ddt\io \rho|\p^3(\nab^\perp\cdot u)|^2+|\p^3(\nab^\perp\cdot B)|^2\dx= -\io \left((B\cdot\nab^\perp)\p^3\nab\cdot u\right)(\p^3\nab^\perp\cdot B)\dx+ \text{ controllable terms},
\end{equation}where the highest order term can be reduced as follows after invoking the first two equations of \eqref{CMHD3}
\[
\begin{aligned}
-(B\cdot\nab^\perp)\p^3\nab\cdot u\lleq&~ \eps^2B\cdot \p^3D_t \nab^\perp p\lleq \eps^2(B_2\p^3D_t\p_1 p-B_1\p^3D_t\p_2 p)\\
\lleq &~-\eps^2\rho(B_2\p^3D_t^2u_1-B_1\p^3D_t^2u_2)-\eps^2(B_2^2+B_1^2)D_t\p^3(\nab^\perp \cdot B),
\end{aligned}
\]where we note that the momentum equations can be written as $\p_1 p=-\rho D_tu_1-B_2(\nab^\perp\cdot B)$ and $\p_2 p=-\rho D_tu_2+B_1(\nab^\perp\cdot B)$. The first term can be treated as in the 3D case, and the contribution of the second term gives an energy term $-\frac12\ddt\io \eps^2|B|^2|\p^3(\nab^\perp\cdot B)|^2\dx$ which also automatically vanishes when we take the incompressible limit $\eps\to 0_+$.
\end{rmk}

It should be emphasized that our analysis is very different from the previous works about well-posedness \cite{1991MHDfirst, Secchi1995, Zhang2021CMHD}. Indeed, these previous works \cite{1991MHDfirst, Secchi1995, Zhang2021CMHD} avoided using div-curl analysis to reduce the normal derivatives of $u$ and $B$. Instead, they controlled the normal derivatives by directly computing the corresponding $L^2$-type estimates which yields non-zero boundary integrals such as $-\int_{\Sigma} \p_3^4 P~\p_3^4v_3\dx'$ and only gives the estimates for $(u,B,\eps p)$, not $(u,B,p)$. Since taking normal derivatives does not preserve the boundary conditions, one has to use the MHD equations to replace $\p_3 P$ and $\p_3 u_3$ by tangential derivatives of the other variables, and then uses Gauss-Green formula to rewrite the boundary integrals into the interior. During this process, more time derivatives might be produced without extra weights of Mach number (e.g., the product of the underlined terms in $-\nab P\to \underline{D_t v}-(\bar{B}\cdot\cnab) B$ and $\p_3v_3\to-\eps^2 D_t p-\underline{\cnab\cdot\bar{v}}$) . However, the continuity equation indicates that one may have to add more Mach number weights to higher-order time derivatives. Thus, following the methods in \cite{1991MHDfirst, Secchi1995, Zhang2021CMHD} might cause a potential of loss of Mach number weight, and one can find related details in \cite[Sect. 2.5]{Zhang2021CMHD}. In other words, these previous works \cite{1991MHDfirst, Secchi1995, Zhang2021CMHD} only revealed that one should trade one normal derivative for two derivatives but ignored the necessity of adding extra weight of Mach number $\eps^2$. One of the advantanges of our method is that all normal derivatives are reduced via div-curl analysis which does not involve any boundary estimates; and every time when we reduce a normal derivative, we never lose weights of Mach number because the continuity equation must be used in the reduction.

\subsection{Reduction of pressure and design of energy functional}\label{sect stat 3}

The main idea of proving the uniform-in-$\eps$ estimates is to repeatedly reduce normal derivatives to tangential derivatives until all derivatives are tangential (with suitable Mach number weights); and reduce spatial derivatives of $p$ to tangential derivatives of $u,B$ until all derivatives on $p$ are time derivatives. Once we only have tangential derivatives, it suffices to mimic the $L^2$ estimates to close the tangential estimates because taking tangential derivatives preserves the boundary conditions. 

From the definition of $D_t$ and $u_3|_{\Sigma}=B_3|_{\Sigma}=0$, we know both $D_t$ and $B\cdot\nab$ are tangential derivatives. Based on this fact and the analysis in section \ref{sect stat 2}, we have the following reductions:
\begin{itemize}
\item [a.] Vorticity: $\nab\times(u,B)\to \eps^2\TT^2 u$.
\item [b.] Divergence: $\nab\cdot(u,B)\to\eps^2\TT p$.
\item [c.] Reduction of pressure: The momentum equation reads $-\nabla (p+|B|^2/2)=\rho D_t u -B\cdot \nab B$ gives the relation $\nabla P\to \TT(u,B)$ and thus $\nab p\to \TT(u,B)$ plus $\nab B$ where $\nab B$ is further reduced by using (a) and (b). This indicates that $\p_t^k\nab p$ should have the same Mach number weight as $\p_t^{k+1}(u,B,S)$.
\item [d.] Tangential estimates: When estimating $E_{4+l}(t)$ (defined in \eqref{energy intro}), $\TT^{\alpha}(u,B)$ is controlled together with $\eps\TT^{\alpha}p$ in the estimates of full tangential derivatives, i.e., when $\lee\alpha\ree=4+l$.
\end{itemize} 

In the above subsection we start with $E_4$ and reduce $\|u,B\|_4$ (part of $E_4$) to $\|\eps^2\TT p\|_{3}$ (still a part of $E_4$) and $\|\eps^2\TT^2 u\|_3$ (part of $E_5$) via the div-curl analysis. The divergence part introduces time derivative, but the number of derivatives does not exceed 4, so we can repeatedly use (c) to reduce the spatial derivatives of $p$ to tangential derivatives of $u,B$ until there is no spatial derivative falling on $p$. Finally, it suffices to control 4-th order time derivatives of $u,B$ and $p$ with suitable Mach number weights. The Mach number weights of these quantities are determined by the relation (d) and the ``preparedness" of initial data. See also diagram 2 below.

For the control of terms that have 5 derivatives, for example the term $\|\eps^2\TT^2 u\|_3$, we repeat the div-curl analysis to get $\|\eps^4\TT^3 p\|_{2}$ (still a part of $E_5$) from the divergence part and $\|\eps^4\TT^4 u\|_{2}$ (part of $E_6$) from the curl part. Again, the divergence part is reduced to $E_5(t)$ which is controlled in a similar way as $E_4(t)$; and then we do further div-curl decomposition for $\|\eps^4\TT^4 u\|_{2}$ which then requires the control of $E_7(t)$. Repeatedly, we finally need to control the $L^2$ norms of $\eps^8\TT^8(u,B)$, which is controlled together with $\eps^8\TT^8 p$ when $\TT^8$ are not purely time derivatives, and $\eps^9\p_t^8 p$ when $\TT^8$ are purely time derivatives. Then the control of $\TT^8$ is completely parallel to the proof of $L^2$ energy conservation because there isn't any normal derivative.

We present the above complicated reduction procedure in the following diagrams. For more details of the reduction procedures, we refer to section \ref{sect tg list}. We also list an example of reduction of $E_{4+l}(t)$ by repeatedly using (c) and (d) when $l=0$. The other cases follow in the same manner.

{\small\begin{equation}\label{diagram1}
\begin{tikzcd}
\|(u,B)\|_4 \arrow[d, "~~", "\text{div}" ] \arrow[r, "~~", "\text{curl}" ]&\|\eps^2\TT^2(u,B)\|_3 \arrow[d, "~~", "\text{div}" ]\arrow[r, "~~", "\text{curl}" ]&\|\eps^4\TT^4(u,B)\|_2 \arrow[d, "~~", "\text{div}" ]\arrow[r, "~~", "\text{curl}" ]&\|\eps^6\TT^6(u,B)\|_1 \arrow[d, "~~", "\text{div}" ]\arrow[r, "~~", "\text{curl}" ]&\|\eps^8\TT^8(u,B)\|_0\arrow[dd, "~~", "\p_t\text{ appears}" ] \\
\|\eps^2\TT p\|_3\arrow[d, "~~", "\p_t\text{ appears}" ] &\|\eps^4\TT^3 p\|_2\arrow[d, "~~", "\p_t\text{ appears}" ] &\|\eps^6\TT^5 p\|_1\arrow[d, "~~", "\p_t\text{ appears}" ] &\|\eps^8\TT^7 p\|_0\arrow[d, "~~", "\p_t\text{ appears}" ] \\
\text{reduction of }E_4\arrow[r, "~~", "~~" ]\arrow[rrd, "~~", "\text{tangential}" ]&\text{reduction of }E_5\arrow[r, "~~", "~~" ]\arrow[rd, "~~", "\text{tangential}" ]&\text{reduction of }E_6\arrow[r, "~~", "~~" ]\arrow[d, "~~", "\text{tangential}" ]&\text{reduction of }E_7\arrow[r, "~~", "~~" ]\arrow[ld, "~~", "\text{tangential}" ]&\text{reduction of }E_8\arrow[lld, "~~", "\text{tangential}" ]\\
&&\text{closed}&&
\end{tikzcd}
\end{equation}
\begin{center}
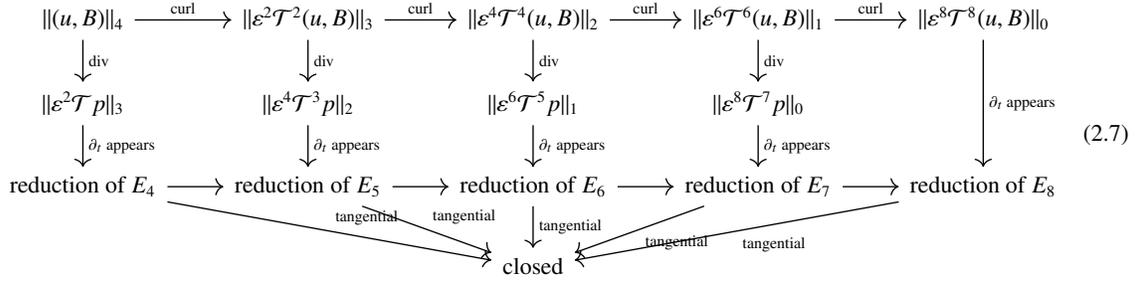
Diagram \eqref{diagram1}: From standard Sobolev space to anisotropic Sobolev space.
\end{center}}

{\small\begin{equation}\label{diagram2}
\begin{tikzcd}
&&E_5(t)&&\\
\|(u,B)\|_4\arrow[urr, "~~", "\text{curl}" ] &\|\p_t(u,B)\|_3\arrow[ur, "~~", "\text{curl}" ]&\|\eps\p_t^2(u,B)\|_2  \arrow[d, "~~", "\text{(d)}" ]\arrow[u, "~~", "\text{curl}" ]&\|\eps^2\p_t^3(u,B)\|_1 \arrow[d, "~~", "\text{(d)}" ]\arrow[ul, "~~", "\text{curl}" ]&\|\eps^3\p_t^4(u,B)\|_0\arrow[d, "~~", "\text{(d)}" ] &\\
\|p\|_4\arrow[ur, "~~","\text{(c)}" ]  & \|\eps\p_t p\|_3 \arrow[ur, "~~","\text{(c)}" ] & \|\eps^2\p_t^2 p\|_2\arrow[ur, "~~","\text{(c)}" ] & \|\eps^3\p_t^3 p\|_1 \arrow[ur, "~~","\text{(c)}" ]  & \|\eps^4\p_t^4 p\|_0\arrow[r, "~~","\text{tangential}" ] &\text{closed}
\end{tikzcd}
\end{equation}
\begin{center}
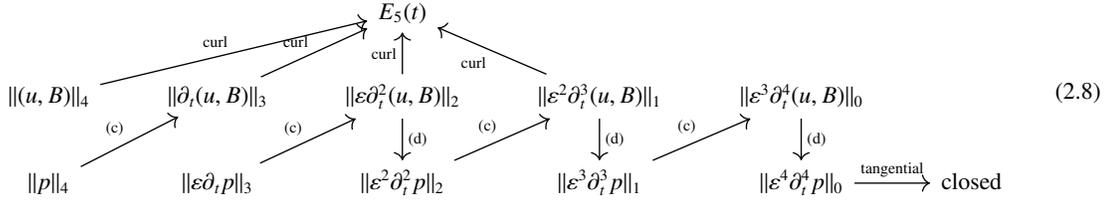
Diagram \eqref{diagram2}: Reduction of $E_4(t)$ via (c) and (d).
\end{center}}

Recall that the vorticity estimates for compressible Euler equations can be closed under the setting of standard Sobolev spaces. As presented in \cite{LuoZhang2022CWWST}, the reduction scheme for Euler equations is merely the control of $E_4(t)$ but does not involve any higher-order energies. In other words, the reduction scheme is finished within the first column of diagram \eqref{diagram1}, or equivalently diagram \eqref{diagram2} with $E_5(t)$ replaced by $E_4(t)$. Compared with Euler equations, the extra thing for compressible ideal MHD is that the vorticity analysis for $E_{4+l}(t)$ requires the control of $E_{4+l+1}(t)$ until there is no normal derivative on $u,B$. That's why we have the first row in diagram \eqref{diagram1}. Then, as presented in each column of diagram \eqref{diagram1}, we need to mimic the scheme in \cite{LuoZhang2022CWWST}, that is, run the scheme presented in diagram \eqref{diagram2}, for each $E_{4+l}(t)$ to close the energy estimates.

		\subsection{Further applications of our method}\label{sect stat 4} 
		 Based on the above analysis, we give a complete and definitive answer to the incompressible limit problem of ideal MHD under the perfectly conducting wall condition when the initial data is well-prepared. Besides, compared with the previous work \cite{JuWangMHDlimit} by Ju and the first author, we give a clear illustration on the ``incompatibility" of the function spaces of well-posedness and thoroughly resolve this issue. It should be noted that, in the uniform-in-$\eps$ estimates, the ``preparedness" of initial data is only used to guarantee $E(0)<\infty$. If the initial data is not ``prepared", that is, $\nab\cdot u_0=O(1)$ is not small and $\p_t u|_{t=0}=O(1/\eps)$ is unbounded, we believe that the uniform-in-$\eps$ estimates can be proven in a similar way after doing suitable technical modifications. Hence, one can expect to generalize our result to the case when the initial data is not ``prepared".

It should also be noted that the reduction procedure presented in diagram \eqref{diagram2} is parallel to the case of compressible water wave, which is presented in the paper \cite[Sect. 2.1.3]{LuoZhang2022CWWST} by Luo and the second author. Indeed, it is exactly the appearance of the magnetic field and the perfectly conducting wall condition that force us to do further vorticity analysis as stated in section \ref{sect stat 2} and presented in diagram \eqref{diagram1}. Therefore, one can expect to extend our results to free-boundary problems. In fact, we believe that our method together with the techniques about free-boundary problems presented in \cite{LuoZhang2022CWWST} can be generalized to study the current-vortex sheets in ideal compressible MHD, which will be presented in a forthcoming paper by the second author.

		\section{Uniform energy estimates}\label{sect uniform}
		\subsection{$L^2$ estimate}\label{sect L2}
		First, we establish $L^2$-energy estimate for \eqref{CMHD3}. Invoking the momentum equation and integrate by parts, we have:
		\begin{equation}
		\begin{aligned}
			\frac12\ddt\io \rho|u|^2\dx=&~\io(B\cdot\nab)B\cdot u\dx-\io u\cdot\nab p \dx-\io u \cdot\nab(\frac12|B|^2)\dx \\
		&+\io\frac12\p_t \rho |u|^2- \rho(u\cdot\nab)u\cdot u\dx\\
		=&-\io (B\cdot\nab)u\cdot B+\io p(\nab\cdot u)\dx +\frac12\io(\nab \cdot u) |B|^2\dx\\
		&+\io\underbrace{\frac12\p_t\rho|u|^2+\frac12\nab\cdot(\rho u) |u|^2}_{=0}\dx.
		\end{aligned}
		\end{equation} Then, inserting the continuity equation and the evolution equation of $B$, we get
		\begin{equation}
		\begin{aligned}
			\frac12\ddt\io \rho|u|^2\dx=&-\io \p_t B\cdot B -\io (u\cdot\nab) B\cdot B -\io |B|^2(\nab\cdot u)  +\frac12\io(\nab \cdot u) |B|^2\dx\\
		&-\io \ffp p\p_t p\dx-\io \ffp p (u\cdot\nab)p \dx\\
		=&-\frac12\ddt\io |B|^2+\ffp |p|^2 \dx+\frac12\io ((\nab\cdot u)  +(\p_t+u\cdot\nab)\ffp) |p|^2\dx.
		\end{aligned}
		\end{equation}  Similarly, we have $\frac12\ddt\io \rho S^2=0$ thanks to $D_t S=0$, so we get the $L^2$ estimate:		
	    \begin{align}\label{L2}
	    \ddt \io \rho|u|^2+|B|^2+\ffp |p|^2+\rho S^2\dx\le P(E_4(t)).
	    \end{align}

		\subsection{Reduction of normal derivatives}\label{sect reduce}
		\subsubsection{Reduction of pressure}\label{sect reduction q}
		We show how to reduce the control of the pressure to that of the velocity and magnetic field when there is at least one spatial derivative on $P$. This follows from using the momentum equation
		\begin{align}\label{mom 3}
			-\nab (p+|B|^2/2) =\rho D_t u -B\cdot\nab B .
		\end{align}
		Then
		\begin{equation} \label{Dq L2}
			\|\nab (p+|B|^2/2)\|_{0}\lesssim \|\rho\|_{\infty} \| D_t u\|_0 + \| B\cdot\nab B\|_0,
		\end{equation}
		where $D_t u = \p_t u +\bar{u}\cdot \cnab u+ u_3\p_3 u$ and $B\cdot \nab B=\bar{B}\cdot \cnab B+ B_3\p_3 B$. Similarly, using \eqref{mom 3}, we also have the reduction for $L^{\infty}$ norm
		\begin{align}\label{Dq Linf}
			\|\nab (p+|B|^2/2)\|_{\infty}\lesssim \|\rho\|_{\infty} \| D_t u\|_{\infty} + \| B\cdot\nab B\|_{\infty}.
		\end{align}
		Recall that $D_t = \p_t +\bar{u}\cdot \cnab + u_3\p_3 $ and $B\cdot \nab =\bar{B}\cdot \cnab + B_3\p_3 $ are both tangential derivatives thanks to $u_3=B_3=0$ on $\Sigma$. The above inequalities \eqref{Dq L2}-\eqref{Dq Linf} show that $\nab P$ is reduced to $\p_t u$, $\TP u$, $\omega(x_3)\p_3 u$, $\TP B$ and $\omega(x_3)\p_3 B$ for some weight function $\omega(x_3)$ vanishing on $\Sigma$. Thus, the momentum equation reduces a \textit{normal} derivative of $P$ to a \textit{tangential} derivative of $u,B$.
		
		Also, one can reduce the Sobolev norm of $\nab P$ in the same way. Let $\TT=\p_t$ or $\TP$ or $\omega(x_3)\p_3$ and $D=\p$ or $\p_t$.  The above estimate yields the control of $\|\TT^{\alpha}D^k \nab p\|_0$ after taking $\TT^{\alpha}D^k$, $k+\lee\alpha\ree\geq 1$, to \eqref{mom 3}. Specifically, at the leading order, $\|\TT^{\alpha}D^k \nab (p+|B|^2/2)\|_0$ is controlled by
		\begin{align} 
			\|\rho\|_{\infty} \| \TT^{\alpha}D^k \TT u\|_0 + \| \TT^{\alpha}D^k \TT B\|_0.
		\end{align}

		As for the fluid pressure $p$, by definition we have $\p_i p=\p_i P-\p_iB\cdot B$. Thus, to control the Sobolev norms of $\nab p$, we still need to do further analysis for $\nab B$, which will be controlled together with $\nab u$.
		
		
		\subsubsection{Div-Curl analysis}\label{sect div curl}

From the analysis of section \ref{sect reduction q}, in order to control the Sobolev norms defined in \eqref{energy intro}, we shall reduce the normal derivatives of $u$ and $B$ and then analyze the tangential derivatives. We use the following lemma  to reduce the normal derivatives of $u,B$ to their divergence and curl.

		\begin{lem}[Hodge elliptic estimates]\label{hodgeTT}
			For any sufficiently smooth vector field $X$ and any real number $s\geq 1$, one has
			\begin{equation}
				\|X\|_s^2\lesssim \|X\|_0^2+\|\nab\cdot X\|_{s-1}^2+\|\nab\times X\|_{s-1}^2 + |X\cdot N|^2_{s-\frac12}.
			\end{equation} 
		\end{lem}
		We start from the control of $E_4(t)$. We apply Lemma \ref{hodgeTT} to $\|\eps^{(k-1)_+}\p_t^k (u,B)\|_{4-k}$ for $0\leq k \leq 3$ to get
		\begin{align}
			\|\eps^{(k-1)_+}\p_t^k u\|_{4-k}^2\lesssim & \|\eps^{(k-1)_+}\p_t^k u\|_0^2 + \|\eps^{(k-1)_+}\nab\times\p_t^k u\|_{3-k}^2 + \|\eps^{(k-1)_+}\nab\cdot \p_t^k u\|_{3-k}^2, \\
			\|\eps^{(k-1)_+}\p_t^k B\|_{4-k}^2\lesssim & \|\eps^{(k-1)_+}\p_t^k B\|_0^2 + \|\eps^{(k-1)_+}\nab\times\p_t^k B\|_{3-k}^2,
		\end{align} where we already use $\nab \cdot B=0$ and the boundary condition $u_3=B_3=0$ on $\Sigma$.

		The $L^2$ norm of $u$ and $B$ has been controlled in \eqref{L2}, while the control of $\|\eps^{(k-1)_+}\p_t^k (u,B)\|_0~(1\leq k\leq 3)$ is parallel to the case of $k=0$ and will be postponed to later sections. The divergence part is reduced to the estimates of $p$ by using the continuity equation and \eqref{ff'}
		\begin{equation}\label{divv}
			\|\eps^{(k-1)_+}\nab\cdot \p_t^k u\|_{3-k}^2\lesssim\left\|\eps^{(k-1)_+ +2} \p_t^k( D_t p )\right\|_{3-k}^2,
		\end{equation}which will be further reduced to the tangenital estimates of $(u,B)$ by using the argument in Section \ref{sect reduction q}. 

		According to the reduction procedure in section \ref{sect reduction q} and (b)-(d) in section \ref{sect stat 3}, the control of $E_4(t)$ will be reduced to the control of $\TT^{\alpha}(u,B)$ and $\p_t^4 p$ with suitable Mach number weights, plus the curl part. Next, we analyze the curl part via the evolution equation of vorticity. 
		
		\begin{lem}\label{curl H4} It holds that
			\begin{align}
				\sum_{k=0}^{3}\ddt \left(\left\|\eps^{(k-1)_+}\nab\times\p_t^k u\right\|_{3-k}^2 +\left\|\eps^{(k-1)_+}\nab\times\p_t^k B\right\|_{3-k}^2+\left\|\eps^{(k-1)_+}\ffp^{\frac12}B\times(\nab\times\p_t^k B)\right\|_{3-k}^2\right) \le P(E_4(t), E_5(t)).
			\end{align}
		\end{lem}
		\begin{proof} First, we take $\nab\times$ in the momentum equation $\rho D_t u- (B\cdot\nab) B=-\nab P  $ to get the evolution equation of vorticity
		\begin{align}\label{curlv eq}
		\rho D_t  (\nab\times u) - (B\cdot\nab) \nab\times B= \rho [D_t,  \nab\times ]u - [B\cdot\nab,\nab\times]B - (\nab\rho)\times (D_t u),
		\end{align}where we notice that the right side only contains the first-order derivatives and does not lose Mach number weight. Note that the equation of state is smooth, so $\p_S \rho$ is bounded.
		\[
		[D_t,\p_i](\cdot)=-\p_iu_j\p_j(\cdot),~~[B\cdot\nab,\p_i](\cdot)=-\p_i B_k\p_k(\cdot),~~\nab\rho=\frac{\p\rho}{\p p}\nab p+\p_S \rho\nab S=O(\eps^2)\nab p+\p_S\rho \nab S.
		\]

		We first prove the case when $k=0$, and the cases for $1\leq k\leq 3$ follow in the same manner. Considering the structure of the evolution equation \eqref{curlv eq} and $\rho$ is bounded below by a positive constant, it suffices to prove that
\[
\frac12\ddt\io\rho\left|\p^3\nab\times u\right|^2+\left|\p^3\nab\times B\right|^2+\left|\ffp^{\frac12} B\times(\p^3\nab\times B)\right|^2\dx\leq P(E_4(t),E_5(t)).
\]

We take $\p^3$ in \eqref{curlv eq} to get 
		\begin{align}\label{curlv H3 eq}
		\rho D_t  \p^3(\nab\times u) - (B\cdot\nab) \p^3\nab\times B=   \underbrace{\p^3(\text{RHS of }\eqref{curlv eq})+[\rho D_t,\p^3](\nab\times u) - [B\cdot\nab,\p^3](\nab\times B)}_{\RR_1},
		\end{align}where the order of derivatives on the right side must be $\leq 4$. Now, standard $L^2$-type estimate yields that
		\begin{align}
			&\frac{1}{2}\ddt \int_{\Omega} \rho|\p^3\nab\times u|^2 \dx=\int_{\Omega} \rho(\p_t \p^3 \nab\times u  )\cdot \p^3 \nab\times u \dx + \frac{1}{2}\int_{\Om} \p_t \rho |\p^3\nab\times u|^2 \dx \notag\\
			=&\int_{\Om}\rho D_t \p^3 \nab\times u \cdot\p^3\nab\times u \dx \notag\\
			&\underbrace{- \int_{\Om}\rho(u\cdot\nab) \p^3\nab\times u\cdot \p^3\nab \times u \dx - \frac{1}{2}\int_{\Om}(\frac{\p\rho}{\p p}\p_t p +\p_S\rho \p_t S)|\p^3\nab\times u|^2 \dx}_{:=\mathcal{I}_1}.
\end{align}

		Then invoking \eqref{curlv H3 eq} gives us the following terms
		\begin{align}
			&\int_{\Om}(\rho D_t \p^3 \nab\times u) \cdot(\p^3\nab\times u) \dx \notag\\
=&\int_{\Om}\left((B\cdot\nab) \p^3 \nab\times B\right)\cdot(\p^3 \nab\times u) \dx \underbrace{+ \int_{\Om} \RR_1\cdot(\p^3 \nab\times u)\dx }_{:=\mathcal{I}_2} \notag\\
			\overset{B\cdot\nab}{=}&-\int_{\Om}(\p^3\nab\times B) \cdot \left(\p^3\nab\times (B\cdot\nab u)\right) \dx \underbrace{- \int_{\Om}(\p^3 \nab\times B) \cdot \left([B\cdot\nab, \p^3 \nab\times]u\right) \dx}_{:=\mathcal{I}_3} + \mathcal{I}_2.
		\end{align}
		Next, we insert the evolution equation of $B$, that is, $D_t B=(B\cdot\nab)u-B(\nab\cdot u)$ to get
		\begin{align}
			&-\int_{\Om}(\p^3\nab\times B) \cdot  \left(\p^3\nab\times (B\cdot\nab u)\right) \dx \notag \\
=&-\int_{\Om}(\p^3 \nab\times B) \cdot (\p^3 \nab\times D_t B) \dx - \int_{\Om}(\p^3 \nab\times B) \cdot  \left(\p^3 \nab\times (B\nab\cdot u) \right)\dx \notag\\
			=&-\frac{1}{2}\ddt\int_{\Omega} |\p^3\nab\times B|^2 \dx \underbrace{- \int_{\Om} \p^3\nab\times B \cdot \left([\p^3\nab\times, D_t]B + (u\cdot\nab)\p^3\nab\times B \right) \dx}_{:=\mathcal{I}_4} \notag\\
			&\underbrace{+\int_{\Om}(\p^3\nab\times B) \cdot \p^3\nab\times ( \ffp B D_t p) \dx}_{:=\mathcal{K}}.
		\end{align}
		
		Based on the concrete form of commutators, a straightforward product estimate for $\mathcal{I}_j$ $(1\le j\le 4)$ gives
		\begin{align}
			\sum_{j=1}^{4}\mathcal{I}_j \le P(E_4(t)).
		\end{align}

		The most difficult term is $\mathcal{K}$ because the highest-order term has 5-th order derivative and thus cannot be controlled by $E_4(t)$ (we omit the terms generated by derivatives falling on $\ffp$ because they must be lower-order and have no loss of $\eps$-weight thanks to \eqref{ff'} and $\p_t S=O(1)$):
		\begin{align*}
		\ffp\p^3\nab\times(B D_t p)=&~\ffp\p^3((\nab D_t p)\times B+D_t p(\nab\times B))\\
		=&-{\ffp}B\times (\p^3\nab D_t p)\underbrace{-{\ffp}[\p^3, B\times](\nab D_t p)+\p^3(\ffp D_t p(\nab\times B))}_{\RR_2},
		\end{align*}where straightforward calculation shows that $\|\RR_2\|_0\leq P(E_4(t))$. We then further analyze the problematic term $-\ffp B\times (\p^3\nab D_t p)$. We invoke a simple vector identity (in 3D) $(B\cdot \nab)B-\frac12\nab|B|^2=-B\times(\nab\times B)$ to rewrite the momentum equation to be
\[
\rho D_t u+B\times(\nab\times B)=-\nab p.
\] So, we commute $D_t$ with $\nab$ and insert this equation into $\eps^2B\times (\p^3\nab D_t p)$ to get 

		\begin{align}
			&-\ffp B\times (\p^3\nab D_t p)=-{\ffp} B\times(\p^3D_t\nab p)-\ffp B\times(\p^3(\nab u_j\p_j p))\notag\\
			= &~\ffp B\times\p^3D_t (\rho D_t u)+{\ffp} B\times \p^3 D_t(B\times(\nab\times B))- \ffp B\times(\p^3(\nab u_j\p_j p))\notag\\
			= &~\ffp\rho B\times( \p^3 D_t^2 u) +\underbrace{{\ffp} B\times D_t(B\times\p^3 (\nab\times B))}_{=\mathcal{K}_1}\notag\\
&\underbrace{+\ffp B\times([\p^3 D_t,\rho] D_t u)+\ffp B\times\left([\p^3, D_t](\nab\times B)+D_t\left([\p^3,B\times](\nab\times B)\right)\right)-\ffp B\times(\p^3(\nab u_j\p_j p))}_{\RR_3},
		\end{align} where $\|\RR_3\|_0\leq P(E_4(t))$ can be proved by using again the concrete forms of these commutators as the order of derivatives does not exceed 4. In particular, in the commutator $D_t\left([\p^3,B\times](\nab\times B)\right)$, the highest-order terms have the form $D_t\left(\p B\times \p^2(\nab\times B)\right)$ or $D_t\left(\p^3 B\times(\nab\times B)\right)$ and contain a time derivative, in which there is no loss of either derivative or $\eps$-weights because there is already an $\eps^2$ factor in $\RR_3$. Now, we analyze the contribution of $\mathcal{K}_1$ in the integral $\mathcal{K}$, which reads
		\begin{align}
		\io \ffp (\p^3\nab\times B) \cdot \left(B\times D_t\left(B\times\p^3 (\nab\times B)\right)\right)\dx.
		\end{align} Then we let $\bd{u}=B$, $\bd{v}=D_t\left(B\times\p^3 (\nab\times B)\right)$ and $\bd{w}=\p^3(\nab\times B)$ in the vector identity $(\bd{u}\times \bd{v})\cdot\bd{w}=-(\bd{u}\times \bd{w})\cdot\bd{v}$ to get
		\begin{align}
		&\io\ffp (\p^3\nab\times B) \cdot \left(B\times D_t\left(B\times\p^3( \nab\times B)\right)\right)\dx=-\io{\ffp} D_t\left(B\times(\p^3 \nab\times B)\right) \cdot \left(B\times(\p^3\nab\times B)\right)\dx,
		\end{align} in which the right side contributes to an $\eps$-weighted energy terms instead of introducing a loss of derivative
		\begin{align}
		&-\io {\ffp} D_t\left(B\times(\p^3 \nab\times B)\right) \cdot \left(B\times(\p^3\nab\times B)\right)\dx\no\\
		=&-\frac12\ddt\io{\ffp}\left|B\times(\p^3 \nab\times B)\right|^2\dx +\underbrace{\frac12\io(\nab\cdot u+{D_t \ffp})\left|B\times(\p^3 \nab\times B)\right|^2\dx}_{\RR_4}.
		\end{align}  The term $\RR_4$ is directly controlled by $\|u\|_{W^{1,\infty}}\|B\|_{L^{\infty}}^2\|B\|_4^2\leq P(E_4(t))$. Therefore, we control the term $\mathcal{K}$ with both $E_4(t)$ and $E_5(t)$ as follows:
	    	\begin{align}
	    	\mathcal{K} + \frac12\ddt\io{\ffp}\left|B\times(\p^3 \nab\times B)\right|^2\dx  \le \|B\|_4\left(P(E_4(t))+\|\rho B\|_{\infty}\|\eps^2 D_t^2 u\|_3\right)\le P(E_4(t),E_5(t)),
	   	 \end{align}
		which gives us the energy estimate
		\begin{align}
				\ddt \left(\|\nab\times u\|_3^2 +\|\nab\times B\|_3^2+\left\|{\ffp}^{\frac12} B\times(\nab\times B)\right\|_3^2\right)\leq P(E_4(t),E_5(t)).
		\end{align}

		Recall that $D_t=\p_t+u\cdot \nab=\p_t+\bar{u}\cdot\cnab+u_3\p_3$ and $u_3|_{\Sigma}=0$ imply that $D_t$ is a tangential derivative. So, the above analysis shows that the curl estimate for compressible ideal MHD cannot be closed in standard Sobolev space, but in fact we can trade one normal derivative $(\nab\times)$ for two tangential derivatives together with square Mach number weight, that is, $\eps^2 \TT^2$. This step naturally introduces the so-called anisotropic Sobolev space and also explains why we add $\eps^2$weight to $E_5(t)$.

		Similarly, we can prove the same conclusion for $\p^{\alpha}\p_t^k$ with $k+|\alpha|=3$ by replacing $\p^3$ with $\eps^{(k-1)_+}\p^{\alpha}\p_t^k$. Indeed, the highest order derivatives in the above commutators do not exceed 4-th order, and there is no loss of Mach number weight because none of the above steps creates negative power of Mach number. Besides, the problematic term can still be analyzed in the same manner. Throughout this section, we omit the terms generated by derivatives falling on $\ffp$.
\[
\eps^{(k-1)_++2}B\times (\p^\alpha\p_t^k\nab D_t p)\lleq -\eps^{(k-1)_+}\rho B\times\p^{\alpha}\p_t^k(\eps^2D_t^2 u)-\eps^{(k-1)_+}{\ffp}B\times D_t\left(B\times(\p^{\alpha}\p_t^k (\nab\times B))\right),
\]where the $L^2$ norms of the omitted lower-order terms are controlled by $P(E_4(t))$. The contribution of the last term in the above equality is again the energy term $-\frac12\ddt\io{\ffp}\eps^{(k-1)_+}\left|B\times(\p^\alpha\p_t^k \nab\times B)\right|^2\dx$. The proof of this follows in the same way if we replace $\p^3$ above by $\eps^{(k-1)_+}\p^{\alpha}\p_t^k$ (using the vector identity $(\bd{u}\times \bd{v})\cdot\bd{w}=-(\bd{u}\times \bd{w})\cdot\bd{v}$. Hence, we conclude the vorticity analysis for $E_4(t)$ with the following inequality
		\begin{align}
			\sum_{k=0}^{3}\ddt \left(\left\|\eps^{(k-1)_+}\nab\times\p_t^k u\right\|_{3-k}^2 +\left\|\eps^{(k-1)_+}\nab\times\p_t^k B\right\|_{3-k}^2+\left\|\eps^{(k-1)_+}{\ffp^{\frac12}}B\times\left(\p_t^k\nab\times B\right) \right\|_{3-k}^2\right) \leq P(E_4(t), E_5(t)).
		\end{align}
		\end{proof}

		The above div-curl analysis shows that it is necessary to control $\|\eps^{2+(k-1)_+}\TT^2 \p_t^k u\|_{3-k}$, which should be analyzed together with $\|\eps^{2+(k-1)_+}\TT^2 \p_t^kB\|_{3-k}$. When $k=3$, $\TT^2\p_t^3$ is a purely tangential derivative and we would like to postpone the tangential estimates to the next subsection. When $0\leq k\leq 2$, we again apply Lemma \ref{hodgeTT}  to $\|\eps^{2+(k-1)_+}\p_t^k\TT^2 (u,B)\|_{3-k}$ to obtain
		\begin{align}
			\|\eps^{2+(k-1)_+}\p_t^k\TT^2 u\|_{3-k}^2 \lesssim&~ \|\eps^{2+(k-1)_+}\p_t^k\TT^2 u\|_0^2 + \|\eps^{2+(k-1)_+}\p_t^k\nab\cdot\TT^2 u \|_{2-k}^2+ \|\eps^{2+(k-1)_+}\p_t^k\nab\times\TT^2 u \|_{2-k}^2,\\
			\|\eps^{2+(k-1)_+}\p_t^k\TT^2 B\|_{3-k}^2 \lesssim&~ \|\eps^{2+(k-1)_+}\p_t^k\TT^2 B\|_0^2 +  \|\eps^{2+(k-1)_+}[\TT^2,\nab\cdot]\p_t^k B \|_{2-k}^2,
		\end{align}
	    where we already use the boundary conditions and the divergence constraint to eliminate the boundary normal traces and $\nab\cdot B$. The divergence part is reduced to the estimates of $p$ by using the continuity equation
	    \begin{align}
	    	\|\eps^{2+(k-1)_+}\p_t^k\nab\cdot \TT^2 u\|_{2-k}^2\leq&~\|\eps^{2+(k-1)_+}\p_t^k\TT^2 \nab\cdot u\|_{2-k}^2 + \|\eps^{2+(k-1)_+}[\TT^2,\nab\cdot]\p_t^k u\|_{2-k}^2\notag\\
	    	{\lesssim}&~\left\|\eps^{4+(k-1)_+} \p_t^k\TT^2( D_t p )\right\|_{2-k}^2 + \left\|\eps^{2+(k-1)_+}[\TT^2,\nab\cdot]\p_t^k u\right\|_{2-k}^2,
	    \end{align}
        where $\|\eps^{2+(k-1)_+}[\TT^2,\nab\cdot]\p_t^k u\|_{2-k}^2$ can be controlled by $\eps^4 E_4(t) \leq \eps^4 E_4(0)+\int_0^t E_5(\tau)\mathrm{d}\tau$. Indeed, such commutators only appear when we commute the third component of $\nab\cdot$ with $\omega \p_3$. We have the following identity which can be shown by induction:
        \begin{equation}\label{Com_p3_TP3}
        	[(\omega\p_3)^k,\p_3](\cdot)= \sum_{\ell\le k-1} c_{k,\ell} (\omega\p_3)^\ell \p_3(\cdot)= \sum_{\ell\le k-1} d_{k,\ell} \p_3(\omega\p_3)^\ell(\cdot),
        \end{equation}
        where $c_{k,\ell}$ and $d_{k,\ell}$ are smooth functions that depend on $k$, $\ell$ and the derivatives (up to order $k$) of $\omega$. For example, we assume $k=0$ and $\TT^2=(\omega\p_3)\TP$, for which the highest-order term in the commutator is $\|\eps^2(\p_3\omega(x_3))\p_3\TP u_3\|_2\leq \eps^4 E_4(t).$ This term can be either controlled by $\eps^4 E_4(0)+\int_0^t E_5(\tau)\mathrm{d}\tau$ or absorbed by $E_4(t)$ when $\eps$ is sufficiently small. Similar estimates apply to the commutator $ \|\eps^{2+(k-1)_+}[\TT^2,\nab\cdot]\p_t^k B \|_{2-k}^2$ so we omit the details. 

   The term $\|\eps^{4+(k-1)_+} \p_t^k\TT^2( D_t p )\|_{2-k}^2$ will be further reduced to the tangenital estimates of $(u,B)$ by using the argument in Section \ref{sect reduction q}. Finally, we need to control the tangential derivatives of $u$ and $B$ (including time derivatives) in $E_5(t)$ and the $L^2$ norm of $\p_t^5 p$ with suitable Mach number weights.
	    
	    Next, we analyze the vorticity term. The proof is parallel to Lemma \ref{curl H4}. 
	    \begin{lem}\label{curl H5} It holds that
	    	\begin{equation}
		\begin{aligned}
	    		&\sum_{k=0}^{2}\ddt\left( \left\|\eps^{2+(k-1)_+}\nab\times \TT^2\p_t^k u\right\|_{2-k}^2 + \left\|\eps^{2+(k-1)_+}\nab\times \TT^2\p_t^k B\right\|_{2-k}^2+\left\|\eps^{2+(k-1)_+}{\ffp^{\frac12}}B\times\left(\nab\times\TT^2\p_t^k B\right) \right\|_{2-k}^2  \right)\\
		\le &~P(E_4(t),E_5(t),E_6(t)).
		\end{aligned}
	    	\end{equation}
	    \end{lem}
	    \begin{proof} For the case of $k=0$, $\alpha_3=\alpha_4=0$, $\TT^\alpha=\p_t^{\alpha_0}\p_1^{\alpha_1}\p_2^{\alpha_2}$ with $\alpha_0+\alpha_1+\alpha_2=2$, we write $\TT^\alpha$ to be $\TT^2$ for convenience and take $\p^2\nab\times\TT^2$ in the momentum equation $\rho D_t u-(B\cdot\nab)B=-\nab (p+\frac12|B|^2) $ to get 
		\begin{equation}
			\rho D_t(\p^2\nab\times \TT^2 u)-(B\cdot\nab)(\p^2 \nab\times \TT^2 B)=[\rho D_t, \p^2\nab\times \TT^2]u - [B\cdot\nab,\p^2\nab\times \TT^2]B.
		\end{equation}Again, it suffices to prove that
\[
\frac12\ddt\io\rho\left|\eps^2\p^2\nab\times \TT^2 u\right|^2+\left|\eps^2\p^2\nab\times \TT^2 B\right|^2+\left|\eps^2{\ffp^{\frac12}} B\times(\p^2 \nab\times \TT^2 B)\right|^2\dx\leq P(E_4(t),E_5(t),E_6(t)).
\]
		We compute that
		\begin{align}
			&\frac{1}{2}\ddt\int_{\Om}\rho |\eps^2\p^2\nab\times \TT^2 u|^2 \dx \notag\\
			=&\int_{\Om}\rho D_t (\eps^2\p^2\nab\times \TT^2 u) \cdot(\eps^2\p^2\nab\times \TT^2 u) \dx\notag\\
& \underbrace{- \int_{\Om} \rho(u\cdot\nab) (\eps^2\p^2\nab\times \TT^2 u)\cdot (\eps^2\p^2\nab\times \TT^2 u) \dx - \int_{\Om}(\frac{\p\rho}{\p p}\p_t p+\p_{S}\rho\p_t S ) |\eps^2\p^2\nab\times \TT^2 u|^2 \dx}_{:=\mathcal{I}_1'} \notag\\
			=&\int_{\Om}\eps^2(B\cdot\nab)(\p^2\nab\times\TT^2 B)\cdot(\eps^2\p^2\nab\times \TT^2 u) \dx\notag\\
& \underbrace{+\int_{\Om} \eps^2 \left( [\rho D_t,\p^2\nab\times \TT^2]u-[B\cdot\nab,\p^2\nab\times\TT^2]B  \right)\cdot(\eps^2\p^2\nab\times \TT^2 u) \dx}_{:=\mathcal{I}_2'} + \mathcal{I}_1'.
			\end{align} Then we integrate $B\cdot\nab$ by parts and invoke the evolution equation of $B$, namely $D_t B=(B\cdot\nab)u-B(\nab\cdot u)$ to get 
			\begin{align}
			&\int_{\Om}\eps^2(B\cdot\nab)(\p^2\nab\times\TT^2 B)\cdot(\eps^2\p^2\nab\times \TT^2 u) \dx\notag\\
			=&-\int_{\Om}\eps^2(\p^2\nab\times\TT^2 B)\cdot\eps^2( \p^2\nab\times\TT^2(B\cdot\nab u) ) \dx- \int_{\Om}\eps^2(\p^2\nab\times\TT^2 B)\cdot\eps^2( [B\cdot\nab,\p^2\nab\times\TT^2]u) \dx \notag\\
			=&-\int_{\Om}\eps^2(\p^2\nab\times\TT^2 B)\cdot\eps^2( \p^2\nab\times\TT^2(D_t B) ) \dx -\int_{\Om}\eps^2(\p^2\nab\times\TT^2 B)\cdot\eps^2( \p^2\nab\times\TT^2(B\nab\cdot u) ) \dx \notag\\
			&- \int_{\Om}\eps^2(\p^2\nab\times\TT^2 B)\cdot\eps^2( [B\cdot\nab,\p^2\nab\times\TT^2]u) \dx\notag\\
			=&-\frac{1}{2}\ddt\int_{\Om}|\eps^2\p^2\nab\times\TT^2 B|^2 \dx\underbrace{+ \int_{\Om}\left(\eps^2\p^2\nab\times\TT^2 B\right)\cdot \left(\eps^2{\ffp} \p^2\nab\times\TT^2(B D_t p)\right) \dx}_{:=\mathcal{K}'}  \notag\\
&\underbrace{- \int_{\Om} \eps^2\p^2\nab\times\TT^2 B \cdot\eps^2\left([\p^2\nab\times\TT^2,D_t]B +(u\cdot\nab)(\eps^2\p^2\nab\times\TT^2 B)+[B\cdot\nab,\p^2\nab\times\TT^2]u\right)\dx}_{:=\mathcal{I}_3'} .
		\end{align}
		A straightforward product estimate for $\mathcal{I}_1'$ gives us
		\begin{equation}
			\mathcal{I}_1'\le P(E_4(t),E_5(t)).
		\end{equation}
		For $\mathcal{I}_2'$, we first consider the commutator $\eps^2[\rho u\cdot\nab,\p^2\nab\times\TT^2]u$. Since,	
		\begin{align}
			[\rho u\cdot\nab,\p^2\nab\times\TT^2]u= \p(\rho u_j)\p^2\TT^2\p_j u_i + \TT(\rho u_j)\p^3\TT\p_j u_i+\text{ lower order terms},
		\end{align}
		we obtain
		\begin{align}
			\eps^2\|\p(\rho u_j)\p^2\TT^2\p_j u_i\|_0 \le&~ \|\p(\rho u)\|_\infty \|\eps^2\TT^2 u\|_3 \le P(E_4(t), E_5(t)),\\
			\eps^2 \|\TT(\rho u_j)\p^3\TT\p_j u_i\|_0 \le &~ \|\TT(\rho u)\|_\infty \|\eps^2\TT\p_j u\|_3 \le P(E_4(t), E_5(t)), ~ j=1,2.
		\end{align}
	    Since $u_3|_{\Sigma}=0$, using fundamental theorem of calculus, we have $u_3(x_3)=u_3(-1)+\int_{-1}^{x_3}\p_3u_3(\xi_3)\mathrm{d}\xi_3$ (W.L.O.G. $x_3<0$), so the length of the interval $[-1,x_3]$ is comparable to the weight function $\omega(x_3)=(1-x_3)(1+x_3)$. We have \[ |\TT(\rho u_3)(x_3)|\leq 0+\int_{-1}^{x_3}\|\p_3\TT(\rho u_3)(\xi_3)\|_{\infty}\mathrm{d}\xi_3  \lesssim \omega(x_3)\|\p_3\TT(\rho u_3)\|_{\infty}. \]

Then we can combine this $\omega(x_3)$ with $\p_3$ to convert a normal derivative to a tangential derivative.
	    \begin{align}
	    	\eps^2 \|\TT(\rho u_3)(x_3)\p^3\TT\p_3 u_i\|_0\lesssim&~ \|\p_3\TT(\rho u_3)\|_{\infty}\|\eps^2\omega(x_3)\p^3\TT\p_3 u\|_0 \notag\\
	    	\le&~\|\p_3\TT(\rho u_3)\|_{\infty} \|\eps^2\p^3\TT\omega(x_3)\p_3 u\|_0 + P(E_4(t),E_5(t))\notag\\
	    	\le&~P(E_4(t),E_5(t)).
 	    \end{align}
	    Thus,
	    \begin{equation}
	    	\|\eps^2[\rho u\cdot\nab,\p^2\nab\times\TT^2]u\|_0\le P(E_4(t),E_5(t)).
	    \end{equation}
		Similarly, since $B_3|_{\Sigma}=0$,
		\begin{equation}
			\|\eps^2[\rho \p_t,\p^2\nab\times\TT^2]u\|_0 + \|\eps^2[B\cdot\nab,\p^2\nab\times\TT^2]B\|_0 \le P(E_4(t),E_5(t)),
		\end{equation}
	    then
	    \begin{equation}
	    	\mathcal{I}_2'\le P(E_4(t),E_5(t)).
	    \end{equation}
		We can also get the estimate of $\mathcal{I}_3'$ in a similar way: 
		\begin{equation}
			\mathcal{I}_3'\le P(E_4(t),E_5(t)).
		\end{equation}
			For $\mathcal{K}'$, we consider the estimate of $\eps^2{\ffp} \p^2\nab\times\TT^2(B D_t p)$. Since,	
		\begin{align}
			&\eps^2{\ffp} \p^2\nab\times\TT^2(B D_t p)\notag\\
			=&~\eps^2{\ffp} \p^2\TT^2(D_t \nab p\times B) \underbrace{+ \eps^2{\ffp} \p^2\TT^2(\nab u\cdot\nab p\times B) + \eps^2{\ffp} \p^2\TT^2(D_t p \nab\times B)}_{:=\mathcal{R}_1'}\notag\\
			=&-B\times\eps^2{\ffp} \p^2\TT^2 D_t \nab p \underbrace{- \eps^2{\ffp}[\p^2\TT^2, B\times](D_t\nab p)}_{\mathcal{R}_2'}+ \mathcal{R}_1'\notag\\
			=&~B\times(\eps^2{\ffp}\p^2\TT^2 D_t(\rho D_t u)) + B\times \eps^2{\ffp}\p^2\TT^2 D_t( B\times(\nab\times B) ) + \mathcal{R}_1'+ \mathcal{R}_2'\notag\\
			=&~B\times(\eps^2{\ffp}\rho \p^2\TT^2 D_t^2 u)  + \underbrace{\eps^2{\ffp} B\times  D_t \left(B\times\p^2\TT^2 (\nab\times B) \right)}_{\mathcal{K}_1'} + \mathcal{R}_1'+ \mathcal{R}_2'\notag\\
			&\underbrace{+B\times \eps^2{\ffp}([\p^2\TT^2D_t,\rho]D_tu)+\eps^2 B\times\left([\p^2\TT^2, D_t](\nab\times B)+D_t\left([\p^2\TT^2,B\times](\nab\times B)\right)\right)}_{\mathcal{R}_3'},
		\end{align}
	    we get that
	    \begin{align}
	    	\|\mathcal{R}_1'+\mathcal{R}_2'+\mathcal{R}_3'\|_0\lesssim& P(E_4(t)) ( \|\eps^2 \TT^2 B\|_3 + \|\eps^2 \TT^3 B\|_2 + \|\eps^4 \TT^2 D_t p\|_{2} + \|\eps^4 \TT^2 u\|_3 ).
	    \end{align} It should be noted that when commuting $D_t$ with $\TP$ or $\p_t$, no extra normal derivative is generated as $D_t$ itself is also a tangential derivative. Indeed, when $j=3$ in $[D_t,\TP_i](\cdot)=-\TP_iu_j\p_j(\cdot)$ for $i=t,1,2$, we have $\TP_i u_3|_{\Sigma}=0$ and so this is still a tangential derivative. 

The analysis of $\mathcal{K}_1'$ is also similar to the analysis of $\mathcal{K}_1$. That is, the contribution of $\mathcal{K}'_1$ in $\mathcal{K}'$ is an energy term plus controllable error terms
		\begin{align}
		&\int_{\Om}\left(\eps^2\p^2\nab\times\TT^2 B\right)\cdot \left(\eps^2{\ffp} B\times  D_t \left(B\times\p^2 (\nab\times \TT^2 B) \right) \right)\dx\no\\
		=&-\frac12\ddt\io\left|\eps^2{\ffp^{\frac12}} B\times\p^2 (\nab\times \TT^2 B)\right|^2\dx+\frac12\io({D_t\ffp}+\nab\cdot u)\left|\eps^2{\ffp^{\frac12}} B\times\p^2 (\nab\times \TT^2 B)\right|^2\dx\no\\
		\lesssim&-\frac12\ddt\io\left|\eps^2{\ffp^{\frac12}} B\times\p^2 (\nab\times \TT^2 B)\right|^2\dx+P(E_4(t))E_5(t),
		\end{align}
       and thus
        \begin{equation}
            \mathcal{K}'+\frac12\ddt\io\left|\eps^2{\ffp^{\frac12}} B\times\p^2 (\nab\times \TT^2 B)\right|^2\dx\le P(E_4(t))\|\eps^4 \TT^2 D_t^2 u\|_2+ P(E_4(t),E_5(t)).
        \end{equation}Here we use the fact that $\TT^2$ commutes with $\nab$ when $\alpha_4=0$. So we obtain the vorticity estimates for $k=0$ and $\alpha_4=0$
        \begin{equation}
		\begin{aligned}
        	&\ddt\left(\left\|\eps^2\nab\times \TT^2 u\right\|_2^2+ \left\|\eps^2\nab\times \TT^2 B\right\|_2^2+\left\|\eps^2{\ffp^{\frac12}} B\times (\nab\times \TT^2 B)\right\|_2^2\right)\\
\le &~P(E_4(t),E_5(t))\|\eps^4 \TT^2 D_t^2 u\|_2\le P(E_4(t),E_5(t),E_6(t)).
		\end{aligned}
        \end{equation}
		
		For the case of $\alpha_4\ge 1$, since $[\p_3,\omega\p_3]\neq0$, we need to reconsider the estimates of commutators. First, we consider commutators: $[\TT^2,\nab\times](\cdot)$ and $[\TT^2,\nab\cdot](\cdot)$.  According to the identity \eqref{Com_p3_TP3}, such commutators contain at most two derivatives and at most one of them is $\p_3$. Thus, we have
		\begin{equation}
			\|\eps^2[\TT^2,\nab\times] u\|_2 + \|\eps^2[\TT^2,\nab\times]B\|_2 + \|\eps^2[\TT^2,\nab\cdot]u\|_2 \le \eps^2\left(E_4(0)+\int_0^t P(E_4(\tau))\mathrm{d}\tau\right).
		\end{equation}
	   Next, for example, we consider the commutator $\eps^2[B\cdot\nab, \p^2\nab\times (\omega\p_3)\TT]B$ in $\mathcal{I}_2'$ (we take $\alpha_4=1$, $\TT^{\alpha}=(\omega\p_3)\p_t^{\alpha_0}\p_1^{\alpha_1}\p_2^{\alpha_2}$ with $\alpha_0+\alpha_1+\alpha_2=1$ without loss of generality),
		\begin{align}
			[B\cdot\nab, \p^2\nab\times (\omega\p_3)\TT]B =&~[B_3\p_3, \omega\p_3^4\TT]B +\text{ lower order terms} \notag\\
			\lleq &~B_3\p_3( \omega\p_3^4\TT B )- \omega\p_3^4\TT(B_3\p_3 B)\notag\\
			\lleq&~B_3(\p_3\omega )\p_3^3\TT B,
		\end{align} where the ``lower order terms" represent the terms that do not have the leading order derivative and the terms that $\p^2\nab\times$ contributes to tangential derivatives instead of $\p_3^3$ as above.

		Since $B_3|_{\Sigma}=0$, we again use the fundamental theorem of calculus to get
		\begin{align}
			\eps^2 \|B_3(\p_3\omega) \p_3^3\TT B\|_0 =&~\|\eps^2 |B_3(x_3)| \p_3^3\TT B\|_0 \notag\\
			\lesssim&~\|\p_3 B_3\|_{\infty}\|\eps^2 \omega(x_3)\p_3^3\TT B\|_0 \notag\\
			\le&~P(E_4(t)).
		\end{align}
		All the other terms in $\mathcal{I}_2'$ and $\mathcal{I}_3'$ can be treated similarly. Finally, we need to consider the term $\mathcal{K}_1'$ because we shall control an extra commutator arising from
		\[
D_t \left(B\times\p^2\TT^2 (\nab\times B)\right) =D_t \left(B\times\p^2 (\nab\times\TT^2 B)\right) +D_t \left(B\times\p^2\left([\TT^2 ,\nab\times] B\right) \right).
		\] Again, using \eqref{Com_p3_TP3}, $[\TT^2 ,\nab\times] B$ contains at most two derivatives and at most one of these two derivatives is $\p_3$. Thus, the highest-order part in the term $D_t \left(B\times\p^2\left([\TT^2 ,\nab\times] B\right) \right)$ contains at most 3 normal derivatives and two tangential derivatives, whose $L^2(\Om)$ can still be bounded by $E_5(t)$ aftering multiplying the weight $\eps^4$.

At this point, we obtain
		\begin{align}
			&\ddt\left(\left\|\eps^2\nab\times \TT^2 u\right\|_2^2 + \left\|\eps^2\nab\times \TT^2 B\right\|_2^2+\left\|\eps^2{\ffp^{\frac12}} B\times (\nab\times \TT^2 B)\right\|_2^2  \right)\no\\
\le&~P(E_4(t),E_5(t))\|\eps^4\TT^4 u\|_2\le P(E_4(t),E_5(t),E_6(t)).
		\end{align}

		For $k=1,2$, we obtain the following energy estimates in similar way as in the proof of Lemma \ref{curl H4}
		\begin{align}
			&\ddt\left(\left\|\eps^{2+(k-1)_+}\nab\times\TT^2\p_t^k u\right\|_{2-k}^2+\left\|\eps^{2+(k-1)_+}\nab\times\TT^2\p_t^k B\right\|_{2-k}^2+\left\|\eps^{2+(k-1)_+}{\ffp^{\frac12}} B\times (\nab\times \TT^2\p_t^k B)\right\|_{2-k}^2\right)\no\\
		\le&~P(E_4(t),E_5(t),E_6(t)).
		\end{align}
	    The proof of Lemma \ref{curl H5} is completed.\end{proof}

The proof of Lemma \ref{curl H5} shows that we need to trade one normal derivative for two more tangentials with Mach number weight $\eps^2$, that is, we need to control $\|\eps^4\TT^4(\p_t^ku,\p_t^kB)\|_{2-k}~(k=0,1,2)$. We again do the analogous div-curl analysis for the case $k=0,1$ and reduce it to $\|\eps^6\TT^6(\p_t^ku,\p_t^kB)\|_{1-k}$, and then repeat once more such that we finally need to control $\|\eps^8\TT^8(u,B)\|_{0}$ which is a purely tangential estimate. To sum up, we can obtain the following estimates by mimicing the proof of Lemma \ref{curl H4} and Lemma \ref{curl H5}.

		\begin{lem}		
	The div-curl analysis and vorticity estimates for $\|\eps^4\TT^4 (u,B)\|_{2}$, $\|\eps^4\p_t\TT^4 (u,B)\|_{1}$ and $\|\eps^6\TT^6 (u,B)\|_{1}$ are listed below:
		\begin{align}
		    &\|\eps^4\TT^4\p_t^k u\|_{2-k}^2\lesssim ~\|\eps^4\TT^4 \p_t^k u\|_0^2 +\|\eps^4\nab\times \TT^4 \p_t^k u\|_{1-k}^2 +\|\eps^4\nab\cdot\TT^4 \p_t^k u\|_{1-k}^2 ,~k=0,1,\\
		    &\|\eps^4\TT^4 \p_t^k B\|_{2-k}^2\lesssim ~\|\eps^4\TT^4 \p_t^k B\|_0^2 +\|\eps^4\nab\times \TT^4 \p_t^k B\|_{1-k}^2+ \|\eps^{4+(k-1)_+}[\TT^4,\nab\cdot]\p_t^k B \|_{1-k}^2,~k=0,1,\\
		&\sum_{k=0,1}\ddt\left( \left\|\eps^4\nab\times \p_t^k\TT^4 u\right\|_{1-k}^2 + \left\|\eps^4\nab\times \p_t^k\TT^4 B\right\|_{1-k}^2+\left\|\eps^4{\ffp^{\frac12}} B\times (\nab\times\p_t^k\TT^4 B)\right\|_{1-k}^2 \right)\no\\
\le &~P(E_4(t),E_5(t),E_6(t))\|\eps^6\TT^6\p_t^ku\|_{1-k},\\
		    &\|\eps^6\TT^6 u\|_1^2\lesssim ~\|\eps^6\TT^6 u\|_0^2 +\|\eps^6\nab\times \TT^6 u\|_0^2 +\|\eps^6\nab\cdot\TT^6 u\|_0^2,\\
		   & \|\eps^6\TT^6 B\|_1^2\lesssim ~\|\eps^6\TT^6 B\|_0^2 +\|\eps^6\nab\times \TT^6 B\|_0^2 + \|\eps^{6}[\TT^6,\nab\cdot] B \|_{0}^2,\\
				&\ddt\left(\left\|\eps^6\nab\times \TT^6 u\right\|_{0}^2 + \left\|\eps^6\nab\times \TT^6 B\right\|_{0}^2  +\left\|\eps^6{\ffp^{\frac12}} B\times(\nab\times \TT^6 B)\right\|_0^2\right)\notag\\
\le &~P(E_4(t),E_5(t),E_6(t),E_7(t))\|\eps^8\TT^8 u\|_{0}.
		\end{align}
		\end{lem}
		Besides, the divergence part is reduced to the estimates of $p$ by using the continuity equation
		\begin{align}
			\|\eps^{4}\p_t^k\nab\cdot \TT^4 u\|_{1-k}^2{\lesssim}&\left\|\eps^{6} \p_t^k\TT^4( D_t p )\right\|_{1-k}^2 + \|\eps^{4}[\TT^4,\nab\cdot]\p_t^k u\|_{1-k}^2, ~~~k=0,1,\notag\\
			\|\eps^{6}\nab\cdot \TT^6 u\|_{0}^2{\lesssim}&\left\|\eps^{8} \TT^6( D_t p )\right\|_{0}^2 + \|\eps^{6}[\TT^6,\nab\cdot] u\|_{0}^2,
		\end{align}
		where the commutators can be controlled
		\begin{align}
			\sum_{k=0,1}\|\eps^{4}[\TT^4,\nab\cdot]\p_t^k u\|_{1-k}^2 {\lesssim} &~\eps^4 E_5(t) \lesssim \eps^4 E_5(0)+\int_0^t E_6(\tau)\mathrm{d}\tau,\\
                     \|\eps^{6}[\TT^6,\nab\cdot] u\|_{0}^2 {\lesssim} &~ \eps^4 E_6(t)\lesssim \eps^4 E_6(0)+\int_0^t E_7(\tau)\mathrm{d}\tau.
		\end{align}
		by using the identity \eqref{Com_p3_TP3}. Moreover, $\left\|\eps^{6} \p_t^k\TT^4( D_t p )\right\|_{1-k}^2$ and $\left\|\eps^{8} \TT^6( D_t p )\right\|_{0}^2$ will be further reduced to the tangenital estimates of $(u,B)$ by using the argument in Section \ref{sect reduction q}.

		\subsubsection{The remaining tangential estimates}\label{sect tg list}
		
		Recall that the entropy $S$ is directly controlled via $D_t S=0$, so we can temporarily ignore that. Let us summarize what kind of tangential derivatives we shall control in order to close the estimates that are uniform in Mach number $\eps$. Again we start with $\|\eps^{(k-1)_+}(\p_t^ku,\p_t^kB)\|_{4-k}$ in $E_4(t)$. We set such Mach number weights based on the following facts:
\begin{itemize}
\item [a.] $H^4$ regularity: Several commutator estimates require the bound for $\|\TP^2(u,B)\|_{L^{\infty}}$. Recall that $H^{\frac{d}{2}+\delta}\hookrightarrow L^{\infty}$, so it would be convenient to choose $H^{2+\lceil\frac{d}{2}+\delta\rceil}$ regularity, that is, $H^4$ for $d=2,3$.
\item [b.] The initial data is well-prepared, so $\|\nab p(0)\|_3$ should be bounded with respect to $\eps$ and we shall add Mach number weights to $\p_t^kp$ when $k\geq 1$.
\item [c.] The $L^2$ estimate in section \ref{sect L2} suggests that $u,B,\eps p$ should share the same weights of Mach number $\eps$ because taking tagential derivatives preserves the boundary conditions.
\item [d.] The reduction procedure in section \ref{sect reduction q} shows that $\nab p$ is converted to $\TT(u,B)$ and $\nab B$. When $\TT$ is a spatial derivative, $\TT(u,B)$ and $\nab B$ should be controlled via div-curl analysis as shown in Lemma \ref{curl H4}. When $\TT=\p_t$, we find that $\p_t^k p,~\p_t^{k+1}(u,B)$ should share the same weights of Mach number.
\end{itemize}
Thus, $E_4(t)$ consists of the following quantities.
\begin{align*}
\|u,B,S,p\|_4,~\|\p_tu,\p_tB,\p_tS,\eps\p_t p\|_3,~\|\eps\p_t^2 u,\eps\p_t^2 B,\eps\p_t^2S,\eps^2\p_t^2p\|_2,\\
\|\eps^2\p_t^3 u,\eps^2\p_t^3 B,\eps^2\p_t^3S,\eps^3\p_t^3p\|_1,~\|\eps^3\p_t^4 u,\eps^3\p_t^4 B,\eps^3\p_t^4S,\eps^4\p_t^4p\|_0.
\end{align*}

The Sobolev norms of $\|\eps^{(k-1)_+}(\p_t^ku,\p_t^kB)\|_{4-k}$ for $0\leq k\leq 3$ are controlled via div-curl analysis. The divergence part can be absorbed by $E_4(t)$ itself because the continuity equation produces extra $\eps^2$ weight and $\eps^2<\eps$ since we assume the Mach number is small. Even if we do not assume $\eps$ is small, we can alternatively control the divergence by repeatedly using the reduction of pressure such that there are only tangential derivatives. The curl part, thanks to the special structure of Lorentz force, is reduced to $\|\eps^{(k-1)_++2}D_t^2\p_t^ku\|_{3-k}$ which is a part of $E_5(t)$. So, it remains to control $E_5(t)$ and the full time derivatives, namely $\|\eps^3\p_t^4 u,\eps^3\p_t^4 B,\eps^4\p_t^4p\|_0$.  See also diagram \eqref{diagram2}.

Then $E_5(t)$ is designed in the same manner as $E_4(t)$:
\begin{align*}
&\|\eps^2\TT^2(u,B,S,p)\|_3~(\text{produced from the vorticity analysis of }\|u,B\|_4),\\
&\|\eps^2\TT^2(\p_tu,\p_tB,\p_tS,\eps\p_tp)\|_2~(\text{produced from the vorticity analysis of }\|\p_tu,\p_tB\|_3),\\
&\|\eps^2\TT^2(\eps\p_t^2u,\eps\p_t^2B,\eps\p_t^2S,\eps^2\p_t^2p)\|_1~(\text{produced from the vorticity analysis of }\|\eps\p_t^2u,\eps\p_t^2B\|_2),\\
&\|\eps^2\TT^2(\eps^2\p_t^3u,\eps^2\p_t^3B,\eps^2\p_t^3S,\eps^3\p_t^3p)\|_0~(\text{produced from the vorticity analysis of }\|\eps^2\p_t^3 u,\eps^2\p_t^3 B\|_1).
\end{align*}

The div-curl analysis of $E_5(t)$ produces $\|\eps^{(k-1)_++4}\TT^2 D_t^2\p_t^ku\|_{2-k}$ which is part of $E_6(t)$. In order to close the energy estimates, we now need to control $E_6(t)$ and also the full time derivatives $\|\eps^2\TT^2(\eps\p_t^3u,\eps\p_t^3B,\eps^2\p_t^3p)\|_0$.

Repeat the above procedure once more, we can see $E_6(t)$ should be designed as
\begin{align*}
&\|\eps^4\TT^4(u,B,S,p)\|_2~(\text{produced from the vorticity analysis of }\|\eps^2\TT^2(u,B)\|_3),\\
&\|\eps^4\TT^4(\p_tu,\p_tB,\p_tS,\eps\p_tp)\|_1~(\text{produced from the vorticity analysis of }\|\eps^2\TT^2(\p_tu,\p_tB)\|_2,\\
&\|\eps^4\TT^4(\eps\p_t^2u,\eps\p_t^2B,\eps\p_t^2S,\eps^2\p_t^2p)\|_0~(\text{produced from the vorticity analysis of }\|\eps^2\TT^2(\eps\p_t^2u,\eps\p_t^2B)\|_1).
\end{align*}	Combining the result of vorticity estimate, it remains to control $\|\eps^6\TT^4 D_t^2\p_t^ku\|_{1-k}$ (part of $E_7(t)$) and also the full time derivatives $\|\eps^4\TT^4(\eps\p_t^2u,\eps\p_t^2B,\eps\p_t^2S,\eps^2\p_t^2p)\|_0$.

Repeat the above procedure once more, we can see $E_7(t)$ should be designed as
\begin{align*}
&\|\eps^6\TT^6(u,B,S,p)\|_1~(\text{produced from the vorticity analysis of }\|\eps^4\TT^4(u,B)\|_2),\\
&\|\eps^6\TT^6(\p_tu,\p_tB,\p_tS,\eps\p_tp)\|_0~(\text{produced from the vorticity analysis of }\|\eps^4\TT^4(\p_tu,\p_tB)\|_1\text{ and {(c)}}.
\end{align*}	 

\noindent\textbf{Discussion about the weights of $p$.} For the pressure $p$, we only have the control of $\|\eps^{k+2l}\TT^\alpha\p_t^k p\|_{4-k-l}$ at this step, which has one more $\eps$-weight than what we defined in \eqref{energy intro} when $k\geq 1,~l\geq 1$ and $\alpha_0<2l$. To replace $\eps^{k+2l}$ by $\eps^{(k-1)_++2l}$ when $l\geq 1$ and $\alpha_0<2l$, we just need to notice that $\TT^{\alpha}$ now must contain at least one tangential spatial derivative, and then we can use the momentum equation \[
-\p_i p =B_k\p_i B_k +\rho D_t u_i-(B\cdot\nab) B_i,~-\omega\p_3 p=B_k\underbrace{\omega\p_3 B_k}_{\text{tangential}}+\omega\rho D_t u_3-(\omega B\cdot\nab) B_3
\]to reduce the control of $\|\eps^{(k-1)_++2l}\TT^\alpha\p_t^kp\|_{4-k-l}$ ($\lee\alpha\ree=2l,~\alpha_0<2l$) to $\|\eps^{(k-1)_++2l}\TT^\beta\p_t^k(u,B)\|_{4-k-l}$ for some $\beta$ satisfying $\lee\beta\ree=2l$ and $\alpha_0\leq \beta_0\leq \alpha_0+1$, which has been included in $E_{4+l}(t)$.

Combining the result of vorticity estimate, it remains to control $\|\eps^6\TT^\alpha(\p_tu,\p_tB,\p_tp)\|_0$ for $\lee\alpha\ree=6,~\alpha_0<6$, $\|\eps^6(\p_t^6u,\p_t^6B,\eps\p_t^7p)\|_0 $ and $\|\eps^8\TT^{\alpha}D_t^2 u\|_0$ for $\lee\alpha\ree=6$. This indicates us to define $E_8(t)$ to be\[
E_8(t)=\sum_{\lee\alpha\ree=8,\alpha_0<8}\left\|\eps^8 \TT^{\alpha}(u,B,S,p)\right\|_{0}^2+\left\|\eps^8 \p_t^8(u,B,S,\eps p)\right\|_{0}^2.
\]

To sum up, the remaining estimates are all tangential estimates, which can be proved by analyzing the $L^2$ estimates after taking the following tangential derivatives to MHD system \eqref{CMHD3}
\begin{align}\label{tg list}
\eps^3\p_t^4,~\eps^4\TT^2\p_t^3,~\eps^5\TT^4\p_t^2,~\eps^6\TT^6\p_t,~\eps^8\TT^8.
\end{align}

		\subsection{Tangential estimates}\label{sect tg}

		\subsubsection{The $\TT^{\alpha}$-differentiated equations}\label{sect tg eq}
		By the div-curl analysis, the crucial step is to study the higher order tangential energy estimate of \eqref{CMHD3}. In particular, we define the following tangential derivatives
		\begin{equation}
			\TT_0=\p_t,\quad \TT_1 = \p_1, \quad \TT_2=\p_2, \quad \TT_4=\omega(x_3) \p_3.
		\end{equation} 
		We take $\TT^{\alpha}$ in the momentum equation $\rho D_t u=-\nab p - B\times(\nab\times B) $ to get 
		\begin{equation}
			\rho D_t(\TT^\alpha u)-(B\cdot\nab)(\TT^{\alpha} B)+ \TT^{\alpha} \nab(p+|B|^2/2)=[\rho D_t, \TT^{\alpha}]u - [B\cdot\nab,\TT^{\alpha}]B,
		\end{equation}
		where $\TT^{\alpha}:=\TT_0^{\alpha_0}\TT_1^{\alpha_1}\TT_2^{\alpha_2}\TT_4^{\alpha_4}$ and $\lee\alpha\ree=\alpha_0+\alpha_1+\alpha_2+0\times2+\alpha_4\le 8$. Accoding to section \ref{sect tg list}, it suffices to consider the following cases
\begin{itemize}
\item $\lee\alpha\ree=\alpha_0=4$ with $\eps^3$ weight,
\item $\lee\alpha\ree=5,~\alpha_0\geq 3$ with $\eps^4$ weight,
\item $\lee\alpha\ree=6,~\alpha_0\geq 2$ with $\eps^5$ weight,
\item $\lee\alpha\ree=7,~\alpha_0\geq 1$ with $\eps^6$ weight,
\item $\lee\alpha\ree=8$ with $\eps^8$ weight.
\end{itemize}

Since these derivatives are all tangential, taking either of them still preserves the boundary conditions. For simplicity, we will only show the details of the following cases (either full spatial derivatives or full time derivatives)
\[
\eps^8\TT^{\alpha}\text{ with }\lee\alpha\ree=8\text{ and }\alpha_0=0,~~\eps^8\p_t^8,~~\eps^{k-1}\p_t^{k}~(4\leq k\leq 7)
\] and the other cases (space-time mixed derivatives) can be proven in the same manner.

		\subsubsection{Tangential energy estimate with full spatial derivatives}\label{sect TT AGU}
		In this subsection we study the spatially-differentiated equations, i.e., the equations obtained by commuting $\TT^\alpha$, $\alpha_0=0$, and $\lee\alpha\ree=8$, with \eqref{CMHD3}. We aim to prove the following estimate
		\begin{prop}\label{tgs8}
			For $\TT^{\alpha}$ with multi-index $\alpha$ satisfying $\alpha_0=0$ and $\lee\alpha\ree=8$, we have the energy inequality:
			\begin{equation}\label{TT spatial 4}
				\sum_{\lee\alpha\ree=8,\alpha_0=0}\ddt \left\|\eps^{8}\TT^{\alpha} (u, B, p)\right\|_0^2 \le P(E(t)).
			\end{equation}
		\end{prop}
		\begin{proof}We first consider the case of $\alpha_4=0$, that is, $\TT^\alpha=\TP_1^{\alpha_1}\TP_2^{\alpha_2}$ with $\alpha_1+\alpha_2=8$. To simplify our notation, we will write $\TP^8$ to represent such $\TT^\alpha$.
		\begin{align}\label{tgs8-0}
			&\frac{1}{2}\ddt\int_{\Om}\rho|\eps^{8}\TP^8 u|^2 \dx\notag\\
			=& \int_{\Om} \eps^{16}\rho\TP^{8}\p_t u\cdot\TP^{8}u \dx +\frac{1}{2}\int_{\Om}\p_t \rho |\eps^{8}\TP^{8} u|^2 \dx \notag\\
			=&\int_{\Om}\eps^{8}\TP^{8}(\rho D_t u) \cdot ( \eps^{8}\TP^{8}u) \dx \underbrace{- \int_{\Om}\eps^{16} \left(\rho \TP^{8}(u\cdot\nab u)  +[\TP^{8},\rho]D_t u + (\frac{\p\rho}{\p p}\p_t p+\p_{S}\rho\p_t S )\TP^{8}u\right)\cdot\TP^{8}u \dx} _{:=\mathcal{J}_1}.
		\end{align}

The first integral above, after invoking the momentum equation and integrating $B\cdot\nab$ by parts, is equal to
		\begin{align}\label{tgs8-1}
\int_{\Om}\eps^{8}\TP^{8}(B\cdot\nab B) \cdot(\eps^{8}\TP^{8}u ) \dx - \int_{\Om}\eps^{8}\TP^{8}(\nab(p+ |B|^2/2)) \cdot(\eps^{8}\TP^{4}u ) \dx.
		\end{align}

For the first term, we integrate $B\cdot\nab$ by parts
		\begin{align}\label{tgs8-2}
			&\int_{\Om}\eps^{8}\TP^{8}(B\cdot\nab B) \cdot(\eps^{8}\TP^{8}u ) \dx \notag\\
			\overset{B\cdot\nab}{=}&-\int_{\Om}\eps^{8}\TP^{8} B \cdot (\eps^{8}  \TP^{8}(B\cdot\nab u)) \dx \underbrace{+\int_{\Om}\eps^{8} [\TP^{8},B\cdot\nab]B \cdot(\eps^{8}\TP^{8}u) \dx-\int_{\Om}\eps^{8}\TP^{8} B \cdot (\eps^{8} [B\cdot \nab,\TP^{8}]u) \dx}_{:=\mathcal{J}_2}.
		\end{align}

For the second term, we integrate by parts and invoke the continuity equation to get
		\begin{align}\label{tgs8-3}
			&- \int_{\Om}\eps^{8}\TP^{8}(\nab(p+ |B|^2/2)) \cdot(\eps^{8}\TP^{8}u ) \dx=\int_{\Om}\eps^{8}\TP^{8}(p+ |B|^2/2)\cdot ( \eps^{8} \TP^{8} \nab\cdot u) \dx\notag\\
			=&\int_{\Om}\eps^{8}\TP^{8}(|B|^2/2)\cdot ( \eps^{8} \TP^{8} \nab\cdot u) \dx - \int_{\Om}\eps^{8}\TP^{8}p \cdot ( \eps^{8} \TP^{8} (\eps^2 D_t p)) \dx,
		\end{align} where the boundary intergral on $\Sigma$ vanishes thanks to the boundary condition $u_3=B_3=0$ on $\Sigma$.

The first integral in \eqref{tgs8-2}, after inserting the evolution equation of $B$, is equal to
		\begin{align}\label{tgs8-4}
			&-\frac{1}{2}\ddt\int_{\Om}|\eps^{8}\TP^{8}B|^2 \dx \underbrace{-\int_{\Om}\eps^{8}\TP^{8}B\cdot( (u\cdot\nab)\TP^{8} B + [\TP^{8},u\cdot\nab ]B ) \dx }_{:=\mathcal{J}_3} \notag\\
			&\underbrace{-\int_{\Om}\eps^{8} \TP^{8}B \cdot (\eps^{8} \TP^{8}(B\nab\cdot u)) \dx }_{:=\mathcal{K}_1} ,
		\end{align}which gives the energy of $B$. The last term on the right side of \eqref{tgs8-3} gives the energy of $\eps p$
		\begin{align}\label{tgs8-5}
			&- \int_{\Om}\eps^{8}\TP^{8}p \cdot ( \eps^{8} \TP^{8} ({\ffp} D_t p)) \dx \notag\\
			{\lleq}&-\frac{1}{2}\ddt\int_{\Om}\left|\eps^{8}{\ffp^{\frac12}}\TP^{8}p\right|^2 \dx \underbrace{-\int_{\Om}\eps^{8}\TP^{8} p \cdot ( \eps^8{\ffp}[\TP^{8},u\cdot\nab]p  ) \dx}_{:=\mathcal{J}_4}.
		\end{align}
		Then we analyze the first term on the right side of \eqref{tgs8-3}. Since $\TP(\frac12|B|^2)=\TP B\cdot B$, we have
		\begin{align}\label{tgs8-6}
			&\int_{\Om}\eps^{8}\TP^{8}(|B|^2/2)\cdot ( \eps^{8} \TP^{8} \nab\cdot u) \dx\notag\\
			=&\underbrace{\int_{\Om}(\eps^{8} B)\cdot (\TP^{8}B)  ( \eps^{8} \TP^{8} \nab\cdot u) \dx }_{:=\mathcal{K}_2} +\underbrace{ \sum_{0<\lee\alpha\ree<8}\int_{\Om}\eps^{8}C_{\alpha} (\TP^{\alpha}B) \cdot (\TP^{8-\alpha}B)  ( \eps^{8} \TP^{8} \nab\cdot u) \dx}_{:=\mathcal{J}_5} 
		\end{align}
			 where $C_{\alpha}$ are some positive constants. Now let us control the commutators $\mathcal{J}_1\sim\mathcal{J}_5$. For $\mathcal{J}_1\sim \mathcal{J}_4$, since  $u_3=B_3=0$ on $\Sigma$, it suffices to analyze one of the following two types of commutators
			\[ \eps^8[\TP^8,u\cdot\nab] f,~~\eps^8[\TP^8,B\cdot\nab ]f.  \]  For example. we expand the first one to find that
		\begin{align*}
			\eps^8[\TP^8,u\cdot\nab] f =\eps^8\TP^8 u\cdot\nab f +\sum_{j=1,2}\eps^8(\TP \bar{u}_j)(\cnab_j\TP^{7} f)+\eps^8(\TP u_3)(\p_3\TP^{7} f)+\sum_{k=2}^{7}\binom{8}{k}\eps^8\TP^{k}u_j\p_j\TP^{8-j}f,
		\end{align*}and it is easy to see that the last term is directly controlled by $E(t)$. For the first term, we have
		\[
			\|\eps^8\TP^8 u\cdot\nab f +\eps^8(\TP\bar{u})\cdot(\cnab\TP^{7} f)\|_0\leq \|\eps^8\TP^8 u\|_0\|\nab f\|_{\infty}+\|\TP u\|_{L^{\infty}}\|\eps^8\TP^8 f\|_{0},
		\]which is directly controlled by $E(t)$. For the third term, note that $\TP u_3|_{\Sigma}=0$, using fundamental theorem of calculus, we have 
		\[
			|\TP u_3(x_3)|\lesssim\omega(x_3)\|\TP\p_3u_3\|_{L^{\infty}},
		\]and thus
		\[
			\|\eps^8(\TP u_3)(\p_3\TP^{7} f)\|_0\lesssim\|\TP\p_3u_3\|_{L^{\infty}}\|\eps^8(\omega\p_3)\TP^7 f\|_0
		\]which is again controlled by $E(t)$. As for $\mathcal{J}_5$, we integrate $\TP$ by parts to get
		\begin{align*}
			\mathcal{J}_5\lesssim \sum_{0<\lee\alpha\ree<8}\|\eps^8\TP^{\alpha+1}B\cdot\TP^{8-\alpha} B)\|_0\|\eps^8\TP^7(\nab\cdot u)\|_0\lesssim E(t)\|\eps^{10}\TP^7 D_t p\|_0\leq  	P(E(t)).
		\end{align*}
		Thus, all the commutators $\mathcal{J}_i~(i=1,\cdots,5)$ are directly controlled
		\begin{align}\label{tgs8-7}
	    	\mathcal{J}_1 +\mathcal{J}_2+\mathcal{J}_3+\mathcal{J}_4+\mathcal{J}_5 \le P(E(t)).
		\end{align}

        		Next, we show that $\mathcal{K}_1$ has cancellation with $\mathcal{K}_2$. We have
	        \begin{align}\label{tgs8-8}
	        	\mathcal{K}_1= & -\sum_{i,j=1}^{3}\int_{\Om} \eps^{8}\TP^{8}B_i \cdot (\eps^8 \TP^{8} (B_i \p_j u_j)) \dx\notag\\
	        	=& \underbrace{-\int_{\Om}\eps^{8} B\cdot \TP^{8}B ( \eps^{8} \TP^{8} \nab\cdot u) \dx}_{=-\mathcal{K}_2} -\sum_{i,j=1}^{3}\int_{\Om} \eps^{8}\TP^{8}B_i \cdot (\eps^8 \TP B_i \TP^{7}\p_j u_j) \dx \notag\\
		&\underbrace{-\sum_{i,j=1}^{3}\sum_{2\leq\lee\alpha\ree\le 8}\int_{\Om} \eps^{8}\TP^{8}B_i \cdot (\eps^8 C_{\alpha} \TP^{\alpha}B_i \TP^{8-\alpha} \p_j u_j) \dx}_{\mathcal{J}_6}.
		\end{align}
		The last term $\mathcal{J}_6$ is again directly controlled: $\mathcal{J}_6 \le P(E(t))$.  The first term cancels with $\mathcal{K}_2$. For the second term, we need to further analyze the case when $j=3$. Thanks to the continuity equation, we have $-\p_3 u_3= {\ffp}D_t p +  \cnab\cdot\bar{u}$ and thus 
		\begin{align}
		&-\sum_{i,j=1}^{3}\int_{\Om} \eps^{8}\TP^{8}B_i \cdot (\eps^8 \TP B_i \TP^{7}\p_3 u_3) \dx=\sum_{i=1}^{3}\int_{\Om} \eps^{8}\TP^{8}B_i \cdot \left(\eps^8 \TP B_i \TP^{7}({\ffp} D_t p +  \cnab\cdot\bar{u})\right)\dx \lesssim P(E(t)).
		\end{align}
		When $j=1,2$, the second term is again directly controlled by $P(E(t))$ because all derivatives are tangential.
		\begin{align}
	        	-\sum_{i=1}^{3}\sum_{j=1,2}\int_{\Om} \eps^{8}\TP^{8}B_i \cdot (\eps^8\TP B_i \TP^{7}\TP_j \bar{u}_j) \dx\leq P(E(t)).
	        \end{align}
		Thus, we conclude that
		\begin{align}\label{tgs8-9}
		    \mathcal{K}_1+\mathcal{K}_2\le P(E(t)).
		\end{align}
		Combining \eqref{tgs8-1}-\eqref{tgs8-9}, we conclude that
		\begin{align}\label{tgs8-10}
			\ddt \left\|\eps^{8}\TP^{8} ({\ffp^{\frac12}} p, u, B)\right\|_0^2 \le P(E(t)).
		\end{align}
		To control $\|\eps^8\TP^8 p\|$, it suffices to invoke the momentum equation $-\TP_i p= B\cdot(\TP_i B)+\rho D_t\bar{u}_i-(B\cdot\nab)B$ and thus we convert $\eps^8\TP^8 p$ to $\eps^8\TP^7\TT(u,B)$ which is part of $E_8(t)$. We then conclude that
		\begin{align}\label{tgs8-11}
			\left\|\eps^{8}\TP^{8} (p, u, B)(t)\right\|_0^2 \le E_8(0)+\int_0^t P(E(\tau))\mathrm{d}{\tau}.
		\end{align}
		
		For the case of $\alpha_4\ge 1$, since $[\p_3,\omega\p_3]\neq0$, we need to reconsider the estimates of commutators. For the commutators of type $[\TT^{\alpha},\p_3](\cdot)$, we can use identity \eqref{Com_p3_TP3}. For example, we need to control the extra terms when commuting $\TT^\alpha$ with $\nab$:
\[
-\io(\eps^{8}\TT^{\alpha}u_i)(\eps^8[\TT^{\alpha},\p_i] P)\dx\text{ and }\io(\eps^{8}[\p_i,\TT^{\alpha}]u_i)(\eps^8\TT^{\alpha} P)\dx.
\] Without loss of generality, we assume $\TT^{\alpha}=(\omega(x_3)\p_3)\TP^7$. In the first integral above, the highest-order term should be
\[ 
-\io(\eps^{8}\TT^{\alpha}u_3)(\eps^8~\p_3\omega~\p_3\TP^7 P)\dx.
\] Now we invoke the momentum equation to replace $-\p_3 P$ with $\rho D_t u_3-(B\cdot\nab)B_3$. Recall that $D_t$ and $B\cdot\nabla$ are both tangential derivatives thanks to $u_3|_{\Sigma}=B_3|_{\Sigma}=0$. So the above integral is controlled by $\|\eps^8\TT^{\alpha}u\|_0(\|\eps^8\TT^{\alpha'}u\|_0+\|\eps^8\TT^{\alpha''}B\|_0)$ plus lower-order terms where $\lee\alpha'\ree=\lee\alpha''\ree=8$. Similarly, for the second integral above, one can invoke the continuity equation, which reads $\p_3u_3=-\cnab\cdot\bar{u}-\eps^2 D_t p$, to finish the control. A slight difference is that we do not have the energy for $\|\eps^8\TT^{\alpha}p\|_0$. Luckily, at this step, $\TT^{\alpha}$ must contain a spatial derivative $\omega(x_3)\p_3$. So, it suffices to again invoke the momentum equation \[ -\omega\p_3 p= \underbrace{\omega\p_3}_{\text{tangential}}(\frac12|B|^2)+\omega\rho D_t u_3-\omega(B\cdot\nabla B_3)\] such that all terms that contain the fluid pressure are replaced by the velocity and the magnetic field.

		As for the other commutators, for example, we re-consider $\eps^{8}[\omega\p_3\TP^{7}, B\cdot\nab]B$ in the analogue of $\mathcal{J}_2$ (W.L.O.G. $\alpha_4=1$). We have
		\begin{align}
			[\omega\p_3\TP^{7}, B\cdot \nab]B =&~[\omega\p_3\TP^{7},B_3\p_3]B +[\omega\p_3\TP^{7},\bar{B}\cdot\cnab]B \notag\\
			=&~\omega\p_3\TP^{7}(B_3\p_3 B)- B_3\p_3( \omega\p_3\TP^{7} B ) +[\omega\p_3\TP^{7},\bar{B}\cdot\cnab]B.
		\end{align} 
		Compared with the case of $\alpha_4=0$, the potential risk is that $\p_3$ may fall on $\omega(x_3)$ and then converts the tangential derivative $\omega \p_3$ to be a normal derivative $\p_3$. So, it suffices to analyze the term $-B_3(\p_3 \omega)(\p_3\TP^7 B)$. Luckily, there is a $B_3$ (replaced by $u_3$ if we commute $\TT^{\alpha}$ with $u\cdot\nab$). Since $B_3|_{\Sigma}=0$, we again use the fundamental theorem of calculus to obtain
		\[
			|B_3(x_3)|\leq \omega(x_3)\|\p_3 B_3\|_{\infty},
		\]and thus 
		\begin{align}
			\eps^{8} \|B_3(\p_3 \omega)(\p_3\TP^7 B)\|_0 \lesssim& \|\p_3B_3\|_{\infty} \|\eps^{8} (\omega\p_3)\TP^{7} B\|_0\le E(t).
		\end{align}
		Other terms arise from commutators can be treated similarly. Therefore, we obtain the energy estimate
		\begin{align}
		\sum_{\lee\alpha\ree=8,\alpha_0=0}\ddt \left\|\eps^{8}\TT^{\alpha} (u, B,p)\right\|_0^2 \le P(E(t)).
		\end{align}
	         The proof is completed.\end{proof}

		\subsubsection{Tangential energy estimate with full time derivatives}\label{sect tt AGU}

		In this subsection we study the time-differentiated equations, i.e., the equations obtained by commuting $\p_t^k$, with \eqref{CMHD3}. We aim to prove the following proposition which gives stronger estimates for full time derivatives than what we need in \eqref{tg list}.
		\begin{prop}\label{tgt8-0} The following tangential estimate for full time derivatives holds:
			\begin{equation}\label{tgt8}
				\ddt\left(\sum_{k=4}^{7} \left\|\eps^{k-1}\p_t^k(u, B,{\ffp^{\frac12}} p)\right\|_0^2+ \left\|\eps^8\p_t^8(u, B, \eps p)\right\|_0^2 \right)\le P(E(t)).
			\end{equation}
		\end{prop}
	    \begin{proof}  Let us first consider the case when $k=8$, that is, $\eps^8\p_t^8$-estimates. We just need to replace $\TP^8$ by $\p_t^8$ in the proof of Proposition \ref{tgs8} and then check if the analogues of the commutators $\mathcal{J}_1\sim\mathcal{J}_6$ can be controlled by $P(E(t))$ or not. From \eqref{tgs8-6} and \eqref{tgs8-8}, we know that the analogue of $\mathcal{J}_5$ and $\mathcal{J}_6$ can be controlled in the same way because $\p_t^8 (u,B)$ and $\TP^8(u,B)$ have the same weights of Mach number. Next, let us analyze the commutators in the analogues of $\mathcal{J}_1\sim\mathcal{J}_3$. For example, we consider
\[
\eps^8[\p_t^8,B\cdot\nab]f,~~f=u\text{ or } B.
\]
This commutator is equal to
\begin{align*}
&\eps^8(\p_t^8B_j\p_jf+8\p_tB_j\p_j\p_t^7f)+\sum_{k=2}^7\eps^8\binom{8}{k}\p_t^k B_j\p_j\p_t^{8-k}f\\
=&~(\eps^8\p_t^8B_j)(\p_j f)+8(\p_t B_j)\p_j(\eps^8\p_t^7 f)+\sum_{k=2}^7\binom{8}{k}(\eps^{k-1}\p_t^k B_j)(\eps^{9-k}\p_j\p_t^{8-k}f),
\end{align*}where the $L^2(\Om)$ norms of the first term and the third term can be directly controlled by $E(t)$. For the second term, when $j=1,2$, it can be directly controlled; when $j=3$, we again use $\p_t B_3|_{\Sigma}=0$ and fundamental theorem of calculus to create a weight function $\omega(x_3)$:
\[
\|(\p_t B_j)\p_j(\eps^8\p_t^7 f)\|\leq \|\p_t \bar{B}\|_{\infty}\|\eps^8\TP B\p_t^7 f\|_0+\|\p_3\p_t B_3\|_{\infty}\|\eps^8(\omega\p_3)\p_t^7 f\|_0\lesssim \sqrt{E_4(t)E_8(t)}\leq E(t).
\]
It remains to analyze the commutator $\mathcal{J}_4$ arising from \eqref{tgs8-5}. The difference is that $\p_t^8 p$ needs one more $\eps$ weight. Luckily, this term is produced because we invoke the continuity equation which automatically gives us $\eps^2$ weight. So, the analogue of $\mathcal{J}_4$ is controlled in the following way:
\begin{equation}
\begin{aligned}
&-\io\eps^8\p_t^8p\left(\eps^{8}{\ffp}[\p_t^8,u\cdot\nab]p\right)\dx=-\io\eps^8{\ffp^{\frac12}}\p_t^8p\left(\eps^{8}{\ffp^{\frac12}}[\p_t^8,u\cdot\nab]p\right)\dx\\
=&-\io\eps^8{\ffp^{\frac12}}\p_t^8p\left(\eps^{8}{\ffp^{\frac12}}(\p_t^8u_j\p_j p+8\p_tu_j\p_j\p_t^8 p)\right)\dx\\
&-\io(\eps^8{\ffp^{\frac12}}\p_t^8p)\left(\sum_{k=2}^7\binom{8}{k}(\eps^{k-1}\p_t^k u_j)(\eps^{9-k}{\ffp^{\frac12}}u_j\p_j\p_t^{8-k}p)\right)\dx\lesssim E(t).
\end{aligned}
\end{equation}

Therefore, we conclude the $\eps^8\p_t^8$-estimates as follows
\begin{equation}\label{tgt8-1}
\ddt \left\|\eps^8\p_t^8(u, B, {\ffp^{\frac12}} p)\right\|_0^2 \le P(E(t)).
\end{equation}

For $4\le k\le 7$, the proof is still parallel to Proposition \ref{tgs8} if we replace $\eps^8\p_t^8$ by $\eps^{k-1}\p_t^k~(4\leq k\leq 7)$. We only need to re-consider the estimates of the commutators arising in the analogues of $\mathcal{J}_1\sim\mathcal{J}_6$. For example, we look at the case $k=7$. Similarly as above, we first consider th $L^2$ estimates of the following commuator
		\[
			\eps^{6}[\p_t^{7},B\cdot\nab]f,~~f=u\text{ or } B,
		\]which is equal to 
		\begin{align*}		
			&\eps^6(\p_t^7B_j\p_jf+7\p_tB_j\p_j\p_t^6f)+\sum_{k=2}^6\eps^6\binom{7}{k}\p_t^k B_j\p_j\p_t^{7-k}f\\
			=&~(\eps^6\p_t^7B_j)(\p_j f)+7(\p_t B_j)\p_j(\eps^6\p_t^7 f)+\sum_{k=2}^6\binom{7}{k}(\eps^{k-1}\p_t^k B_j)(\eps^{7-k}\p_j\p_t^{7-k}f).
		\end{align*}Again, the first term and the third term are directly controlled by $E(t)$. For the second term, using $\p_t B_3|_{\Sigma}=0$ and fundamental theorem of calculus, we then obtain 
		\[
			\|(\p_t B_j)\p_j(\eps^6\p_t^7 f)\|_0\leq \|\p_t \bar{B}\|_{\infty}\|\eps^6\TP\p_t^6f\|_0+\|\p_3\p_tB_3\|_{\infty}\|\eps^6(\omega\p_3)\p_t^6f\|_0\lesssim\sqrt{E_4(t)E_7(t)}\lesssim E(t).
		\]
As for the analogue of $\mathcal{J}_4$, the continuity equation gives us extra $\eps^2$ weight such that the estimate is uniform in Mach number:
		\begin{equation}
		\begin{aligned}
			&-\io\eps^6\p_t^7p\left(\eps^{6}{\ffp}[\p_t^7,u\cdot\nab]p\right)\dx=-\io\eps^6{\ffp^{\frac12}}\p_t^6{\ffp^{\frac12}}p\left(\eps^{7}[\p_t^7,u\cdot\nab ]p\right)\dx\\
			=&-\io\eps^6{\ffp^{\frac12}}\p_t^7p\left(\eps^{6}{\ffp^{\frac12}}(\p_t^7u_j\p_j p+7\p_tu_j\p_j\p_t^6 p)\right)\dx\\
			&-\io(\eps^6{\ffp^{\frac12}}\p_t^7p)\left(\sum_{k=2}^6\binom{7}{k}(\eps^{k-1}\p_t^k u_j)(\eps^{7-k}{\ffp^{\frac12}}u_j\p_j\p_t^{8-k}p)\right)\dx\lesssim E(t).
		\end{aligned}
		\end{equation}
Therefore, we obtain the energy estimate
		\begin{align}
			\ddt \left\|\eps^6\p_t^7(u, B, {\ffp^{\frac12}} p)\right\|_0^2\le P(E(t)).
		\end{align}

The case when $4\leq k\leq 6$ can be proved in the same way, so we conclude the following energy estimate for full time derivatives
		\begin{align}
			\ddt\sum_{k=4}^{7} \left\|\eps^{k-1}\p_t^k(u, B, {\ffp^{\frac12}} p)\right\|_0^2\le P(E(t)).
		\end{align}
		The proof is completed.
		\end{proof}

		\subsubsection{Space-time mixed tangential derivatives}

		We still need to prove the tangential estimates for $\eps^{8}\TT^{\alpha}~(0<\alpha_0<8)$. Indeed, they can be proved in the same way as Proposition \ref{tgs8} and Proposition \ref{tgt8-0} because of the following reasons
		\begin{itemize}
			\item The derivatives $\eps^{8}\TT^{\alpha}~(0<\alpha_0<8)$ are tangential, so they preserve the boundary conditions and thus all boundary integrals must vanish.
			\item The commutators arising in the proof of Proposition \ref{tgs8} and Proposition \ref{tgt8-0} do not produce extra time derivatives, so the weight $\eps^8$ is enough to close the estimate.
			\item When $\alpha_4\neq 0$, that is, $(\omega\p_3)^{\alpha_4}$ appears in $\TT^\alpha$, it suffices to use the same strategy as presented in Proposition \ref{tgs8}. Indeed, we only commute $(\omega\p_3)$'s with either  $u\cdot\nab$ or $B\cdot\nab$, so there must a $u_3$ or $B_3$ appearing when $\p_3$ falls on $\omega$. Thus, we can use the boundary conditions for $u_3,~B_3$ to reproduce the weight function $\omega(x_3)$.
		\end{itemize}
		Based on the above analysis, we conclude the tangential estimates in following proposition.
		\begin{prop}\label{tgst8}
			We have the energy inequality:
			\begin{equation}\label{Tt 8}
			\begin{aligned}
			\ddt \bigg(	&\sum_{\lee\alpha\ree=8,\alpha_0<8}\left\|\eps^{8}\TT^{\alpha} (u, B, p)\right\|_0^2 + \left\|\eps^8\p_t^8(u, B, {\ffp^{\frac12}} p)\right\|_0^2 \\
			&+\sum_{l=0}^3\sum_{\lee\alpha\ree=2l,\alpha_0<2l}\left\|\eps^{3+l}\TT^{\alpha}\p_t^{4-l} (u, B, p)\right\|_0^2+\sum_{l=0}^3\left\|\eps^{3+l}\p_t^{4+l} (u, B,{\ffp^{\frac12}} p)\right\|_0^2\bigg)\le P(E(t)).
			\end{aligned}
			\end{equation}
		\end{prop}

		\subsection{Uniform estimates}
		Since $D_t S=0$, we can easily prove the estimates for $S$ by directly commuting $D_t$ with $\p_*^{\alpha}$ and corresponding weights of Mach number. The proof does not involve any boundary term because we do not integrate by parts, so we omit the details.
		\begin{equation}
			\ddt \sum_{l=0}^4\sum_{\lee\alpha\ree=2l}\left\|\eps^{(k-1)_++2l} \TT^{\alpha}\p_t^kS\right\|_{4-k-l}^2\leq P(E(t)).
		\end{equation}

		Combining the tangential estimates presented in Propositions \ref{tgs8}-\ref{tgst8} with the div-curl analysis in section \ref{sect div curl}, the reduction of pressure in section \ref{sect reduction q}, the $L^2$ estimates and the summary in section \ref{sect tg list}, we can get the following energy inequality
		\begin{equation}
			\ddt E(t) \le P(E(t)),
		\end{equation}where $E(t)$ is defined by \eqref{energy intro}.
		Since the right side of the energy inequality does not rely on $\eps$, we can use Gr\"onwall's inequality to prove that there exists some $T>0$ independent of $\eps$ such that
		\begin{equation}
			\sup_{t\in[0,T]} E(t) \le P(E(0)).
		\end{equation}
		Theorem \ref{main thm, well data} is proven.

		\subsection{The case of 2D MHD}\label{sect 2D}
		
		The previous sections are devoted to the incompressible limit of 3D MHD. Now we explain how to modify the proof such that it is valid for 2D MHD flows. The proof for the 2D case is essentially the same as the 3D case: we only need to do a slight modification in the vorticity analysis because the vorticity in 2D is a scalar function, not a vector function. First, lemma \ref{hodgeTT} should be modified as below
		\begin{lem}[Hodge elliptic estimates]\label{hodge2D}
			For any sufficiently smooth vector field $X\in\R^2$ and any real number $s\geq 1$, one has
			\begin{equation}
				\|X\|_s^2\lesssim \|X\|_0^2+\|\nab\cdot X\|_{s-1}^2+\|\nab^\perp\cdot X\|_{s-1}^2 +|X\cdot N|^2_{s-\frac12},
			\end{equation} where $\nab=(\p_1,\p_2)$ and $\nab^\perp=(-\p_2,\p_1)$.
		\end{lem} Taking $\nab^\perp\cdot$ in the momentum equation, we obtain the analogue of \eqref{curlv eq} as
		\begin{equation}\label{curlv eq 2D}
			\rho D_t(\nab^\perp \cdot u)-(B\cdot\nab)(\nab^\perp \cdot B)=\rho[D_t,\nab^\perp\cdot] u-[B\cdot\nab,~\nab^\perp\cdot ] B-(\nab^\perp\rho)\cdot(D_t u),
		\end{equation}which has the same structure as \eqref{curlv eq}. Thus, one can follow the analysis in section \ref{sect div curl}. The only slight difference is the treatment of Lorentz force term. We take $\p^3$-estimate of $\nab^\perp\cdot(u,B)$ for an example. Following the proof of lemma \ref{curl H4}, we need to control the following integral
		\begin{equation}
			\mathcal{K}'':=\int_{\Om}\p^3(\nab^\perp\cdot B)~\p^3\nab^\perp\cdot ( {\ffp} B D_t p) \dx,
		\end{equation}where the problematic term is 
$${\ffp}(B\cdot\nab^\perp)\p^3D_t p={\ffp}(-B_1\p^3\p_2D_t p+B_2\p^3\p_1D_tp)$$
after inserting the continuity equation. Commuting $\nab^\perp$ with $D_t$, we have
		\begin{align*}
			&{\ffp}(-B_1\p^3\p_2D_t p+B_2\p^3\p_1D_tp)\\
			=&~{\ffp}\left(-B_1\p^3D_t(\p_2p)+B_2\p^3D_t(\p_1 p)\right)-{\ffp}\left(B_1\p^3(\p_2u_j\p_j p)-B_2\p^3(\p_1u_j\p_j p)\right),
		\end{align*}where the $L^2(\Om)$ norm of the second term is directly controlled by $E_4(t)$. For the first term, we again invoke the momentum equation, which in 2D reads
		\begin{align*}
		&-\p_1p=\rho D_t u_1-B_1\p_1B_1-B_2\p_2B_1+\p_1(|B|^2/2)=\rho D_tu_1+B_2(\p_1B_2-\p_2 B_1)=\rho D_t u_1+B_2(\nab^\perp\cdot B),\\
		&-\p_2p=\rho D_t u_2-B_1\p_1B_2-B_2\p_2B_2+\p_2(|B|^2/2)=\rho D_tu_1-B_1(\p_1B_2-\p_2 B_1)=\rho D_t u_2-B_1(\nab^\perp\cdot B).
		\end{align*} Now, we have
		\begin{align*}
		&{\ffp}(-B_1\p^3\p_2D_t p+B_2\p^3\p_1D_tp)\\
		=&~{\ffp}\rho (B_1 \p^3 D_t^2 u_2-B_2 \p^3D_t^2 u_1)\underbrace{-{\ffp} (B_1^2+B_2^2) \p^3D_t(\nab^\perp\cdot B)}_{{\ffp}|B|^2\text{ has definite sign}}+~L^2(\Om)\text{-controllable terms},
		\end{align*}and thus
		\begin{equation}
		\begin{aligned}
		\mathcal{K}''\lleq& \io\p^3(\nab^\perp\cdot B) {\ffp}\rho (B_1 \p^3 D_t^2 u_2-B_2 \p^3D_t^2 u_1)\dx-\int_{\Om}\p^3(\nab^\perp\cdot B)~{\ffp} |B|^2 \p^3D_t(\nab^\perp\cdot B) \dx\\
		\lleq&\io \p^3(\nab^\perp\cdot B){\ffp}\rho (B_1 \p^3 D_t^2 u_2-B_2 \p^3D_t^2 u_1)\dx\\
		&-\frac12\ddt\io \left|{\ffp^{\frac12}}|B|\p^3(\nab^\perp\cdot B)\right|^2\dx+\frac12\io (\nab\cdot u)\left(\left|{\ffp^{\frac12}}|B|\p^3(\nab^\perp\cdot B)\right|^2\right)\dx,
		\end{aligned}
		\end{equation}where the last term on the right side can be directly controlled, and the first term requires the control of $\|{\ffp} D_t^2 u\|_3$.  The conclusion of Lemma \ref{curl H4} is then modified to be
		\begin{align}
				\frac12\ddt\io \rho\left|\p^3(\nab^\perp\cdot u)\right|^2 +(1+{\ffp}|B|^2)\left|\p^3(\nab^\perp\cdot B)\right|^2\dx\leq P(E_4(t))+E_5(t).
		\end{align}

		\section{Incompressible limit} \label{sect limit}
		This section is devoted to showing that we can pass the solution of \eqref{CMHD3} to the incompressible limit \eqref{IMHD} provided that we have the convergence for initial data.  In other words, we study the behavior of the solution of \eqref{CMHD3} as the Mach number $\eps$ tends to $0$. Recall that the Mach number $\eps$ is defined in Section \ref{sect mach number definition}.

		The energy estimate presented in Theorem \ref{main thm, well data} implies the boundedness of $\sum\limits_{k=0,1}\|\p_t^k(u,B,S)\|_{4-k}$ uniformly in $\eps$ within the time interval $[0,T]$. Thus, by Aubin-Lions Lemma, up to a subsequence, we have
		\begin{equation}
		\begin{aligned}
			(u,B,S) \to (u^0,B^0,S^0) \quad & \text{weakly-* in } L^{\infty}([0,T];H^4(\Om))\text{ and strongly in } C([0,T];H^{4-\delta}(\Om))
		\end{aligned}
		\end{equation}
		for $\delta>0$. Moreover, by using the continuity equation $\eps^2 D_t p + \nab\cdot u=0$, we have
		\begin{equation}
			\nab\cdot u \to \nab\cdot u^0=0 \quad \text{in } L^{\infty}([0,T];H^3(\Om)) \text{ and strongly in } C([0,T];H^{3-\delta}(\Om))
		\end{equation}
		for $\delta>0$ because $\|\eps \p_{t,x}p\|_3$ and $\|u\|_4$ are uniformly bounded in $[0,T]$.
		
		Now, following the argument in Klainerman-Majda \cite[Theorem 1]{Klainerman1981limit}, there exists a function $\pi$ satisfying $\nab\pi\in C([0,T];H^3(\Om))$ such that
		\begin{enumerate}
		\item The following convergence result holds \[\rho^{-1}\nab p\to \varrho^{-1}\nab \pi\quad \text{weakly-* in } L^{\infty}([0,T];H^3(\Om)) \text{ and strongly in } C([0,T];H^{3-\delta}(\Om)).\]
		\item $(u^0,B^0,S^0)\in L^{\infty}([0,T];H^4(\Om))$ solves the incompressible MHD equations together with a transport equation
		\begin{equation}
			\begin{cases}
				\varrho(\p_t u^0 + u^0\cdot\nab u^0) -B^0\cdot\nab B^0+ \nab (\pi+\frac12|B^0|^2) =0&~~~ \text{in}~[0,T]\times \Omega,\\
				\p_t B^0+ u^0\cdot\nab B^0 -B^0\cdot\nab u^0=0 &~~~ \text{in}~[0,T]\times \Omega,\\
				\nab\cdot u^0=\nab\cdot B^0=0&~~~ \text{in}~[0,T]\times \Omega,\\
				\p_t S^0+u^0\cdot\nab S^0=0&~~~ \text{in}~[0,T]\times \Omega,\\
				u_3^0=B_3^0=0&~~~\text{on}~[0,T]\times\Sigma,\\
				(u^0,B^0,S^0)|_{t=0}=(u_0^0, B_0^0,S_0^0).
			\end{cases}
		\end{equation}
		Here $\varrho$ satisfies $\p_t\varrho+u^0\cdot\nab\varrho=0,$ with initial data $\varrho_0:=\rho(0,S_0^0)$.
		\end{enumerate}
 Moreover, the uniqueness of the limit function implies that the convergence holds as $\eps\to 0$ without restricting to a subsequence. Theorem \ref{main thm, limit} is then proven.
 
 \paragraph*{Data avaliability.} Data sharing is not applicable as no datasets were generated or analyzed during the current study.

\subsection*{Ethics Declarations}
\paragraph*{Conflict of interest.} On behalf of all authors, the corresponding author states that there is no conflict of interest.

\begin{appendix}
\section{The proof of local well-posedness in $H_*^8$}\label{apdx LWP}
In the appendix, we prove the local well-posedness of system \eqref{CMHD3} with initial data in $H_*^8(\Om)$ satisfying the assumption of Theorem \ref{main thm, well data}. This can be done by standard Picard iteration. We start with $(u^{(0)}, B^{(0)}, \rho^{(0)}, S^{(0)})=(\vec{0},\vec{0},1,0)$ and inductively define $(u^\nnr,B^\nnr, P^\nnr, S^\nnr)$ for each $n\in\N$ provided that we already have $\{(u^{(k)}, B^{(k)}, P^{(k)}, S^{(k)})\}_{k\leq n}$ by the following linear system with variable coefficients only depending on $\{(u^{(k)}, B^{(k)}, P^{(k)}, S^{(k)})\}_{k\leq n}$:
\begin{equation}\label{CMHDll}
\begin{cases}
\rho^\nnn(\p_t + u^\nnn\cdot\nab) u^\nnr - (B^\nnn\cdot\nab) B^\nnr + \nab P^\nnr = 0 &\text{ in }[0,T]\times\Om,\\
\ffp^\nnn(\p_t + u^\nnn\cdot\nab)  P^\nnr - \ffp^\nnn(\p_t + u^\nnn\cdot\nab)  B^\nnr \cdot B^\nnn + \nab\cdot u^\nnr = 0 &\text{ in }[0,T]\times\Om,\\
(\p_t + u^\nnn\cdot\nab) B^\nnr - (B^\nnn\cdot\nab) u^\nnr + B^\nnn(\nab\cdot u^\nnr) = 0 &\text{ in }[0,T]\times\Om,\\
(\p_t + u^\nnn\cdot\nab) S^\nnr  =0 &\text{ in }[0,T]\times\Om,\\
u_3^\nnr|_{\Sigma}=0 &\text{ on }[0,T]\times\Sigma,\\
(u^\nnr,B^\nnr, P^\nnr, S^\nnr)|_{t=0}=(u_0,B_0,P_0,S_0).
\end{cases}
\end{equation} The basic state $(u^\nnn, B^\nnn, \rho^\nnn, S^\nnn)$ satisfies the following conditions:
\begin{enumerate}
\item (Hyperbolicity assumption) $\rho^\nnn\geq \bar{\rho}_0>0$ for some constant $\bar{\rho}_0$. It is determined by the equation of state $p^\nnn=p^\nnn(\rho^\nnn, S^\nnn)$ where $p^\nnn:=P^\nnn-\frac12|B^\nnn|^2$.
\item (Boundary constraint for the magnetic field) $B_3^\nnn|_{\Sigma}=0$. (Note that $\nab\cdot B^\nnn=0$ no longer propagates from the initial data for the linearized problem!)
\item (Slip condition) $u_3^\nnn|_{\Sigma}=0$.
\item (Equation of state) $\ffp^\nnn:=\frac{\p \ff^{\nnn}}{\p p}(p^\nnn, S^\nnn)$.
\end{enumerate}
After solving $(u^\nnr,B^\nnr, P^\nnr, S^\nnr)$ from \eqref{CMHDll}, we can define $p^\nnr:=P^\nnr-\frac12|B^\nnr|^2$ and define $\rho^\nnr$  from the equation of state $p^\nnr=p^\nnr(\rho^\nnr, S^\nnr)$. 

System \eqref{CMHDll} is a linear, first-order, symmetric hyperbolic system with characteristic boundary conditions. Under the conditions $B_3^\nnn|_{\Sigma}=u_3^\nnn|_{\Sigma}=0$, one can use the argument in \cite{Rauch1985} to compute that the correct number of boundary conditions of \eqref{CMHDll} is 1, that is, $u_3^\nnr|_{\Sigma}=0$. The local existence of \eqref{CMHDll} in $L^2(0,T_0;\Om)$ for some $T_0>0$ can be established by standard hyperbolic PDE theory, e.g., Rauch \cite[Theorem 9]{Rauch1985}.Moreover, since our initial data belongs to $H^8(\Om)$ and satisfies the compatibility conditions up to 7-th order, then \cite[Theorem 2.1']{Secchi1996} implies that this solution is a strong solution and satisfies $\p_t^j(u,B,P,S)\in C([0,T_0];H_*^{8-j}(\Om))$ for $j\leq 8$.

We also need to verify that the solution $B_3^{\nnr}$ vanishes on $\Sigma$ provided $B_3^\nnr|_{t=0}=0$ on $\Sigma$. In fact, the proof of this step is rather easy by restricting the third component of the evolution equation of $B^\nnr$ onto $\Sigma$. Since $u_3^\nnn=B_3^\nnn=0$ on $\Sigma$, we have that $(\p_t +\bar{u}^\nnn\cdot\cnab) B_3^\nnr = 0$ on $\Sigma$. Then standard $L^2(\Sigma)$ estimate leads to the desired result.
\begin{rmk}
It should be noted that there is no need to add modification terms to $B_3^\nnr$ to preserve this constraint as in free-boundary problems \cite{Trakhinin2008CMHDVS, TW2020MHDLWP, TW2021MHDCDST} because now $\Om$ is fixed and its boundary is flat.
\end{rmk}

Now, it remains to prove the following two things:
\begin{itemize}
\item Uniform-in-$n$ estimates for the linear problem \eqref{CMHDll}. We will still use the energy functional $E(t)$ with $\ffp$ replaced by $\ffp^\nnn$. Also, we will see that the estimates are also uniform in $\eps>0$.
\item Strong convergence of the sequence $(u^\nnn, B^\nnn, P^\nnn, S^\nnn)$ as $n\to \infty$. Once we establish this convergence result, the limit system of \eqref{CMHDll} is exactly the original MHD system \eqref{CMHD3}. That is, the local existence of \eqref{CMHD3} is proved. The uniqueness can be established in the same way and thus we will omit the proof.
\end{itemize}

For simplicity of notations, given any $n\in\N$, we denote $(u^\nnr,B^\nnr, P^\nnr, S^\nnr, p^\nnr, \rho^\nnr)$ by $(u,B,P,S,p,\rho)$  and denote the basic states $(u^\nnn, B^\nnn, P^\nnn, S^\nnn, p^\nnn, \rho^\nnn)$ by $(\ur ,\br,\mathring{P},\sr, \pr,\rhor)$. We also write $D_t:=\p_t+u\cdot\nab$ and $\Dtr:=\p_t+\ur\cdot\nab$. Thus, the linear problem \eqref{CMHDll} becomes
\begin{equation}\label{CMHDll0}
\begin{cases}
\rhor \Dtr u - (\br\cdot\nab) B + \nab P = 0 &\text{ in }[0,T]\times\Om,\\
\ffpr \Dtr  P - \ffpr\Dtr  B \cdot \br + \nab\cdot u = 0 &\text{ in }[0,T]\times\Om,\\
\Dtr B - (\br\cdot\nab) u + \br (\nab\cdot u) = 0 &\text{ in }[0,T]\times\Om,\\
\Dtr S  = 0 &\text{ in }[0,T]\times\Om,\\
u_3|_{\Sigma}=0 &\text{ on }[0,T]\times\Sigma,\\
(u,B,P,S)|_{t=0}=(u_0,B_0,P_0,S_0),
\end{cases}
\end{equation}under the assumption $\ur_3|_{\Sigma}=\br_3|_{\Sigma}=0$.

\subsection{Uniform estimates for the linearized problem}
We assume the basic state $(\ur ,\br,\mathring{P},\sr, \pr,\rhor)$ satisfies the following bound: There exists a constant $K_0>0$ and a time $T_0>0$ (both independent of $\eps>0$) such that
\begin{align}\label{energy Err}
\sup_{0\leq T\leq T_0}\sum_{l=0}^4\sum_{\lee\alpha\ree=2l}\sum_{k=0}^{4-l}\left\|\eps^{(k-1)_++2l}  \TT^{\alpha}\p_t^k\left(\ur,\br,\sr,\dot{\ffp}^{\frac{(k+\alpha_0-l-3)_+}{2}}\mathring{P}\right)\right\|_{4-k-l}^2\leq K_0,
\end{align} where $\dot{\ffp}$ is defined by replacing the superscript $\nnn$ by $\nnl$ in $\ffpr$. For the solution $(u,B,P,S,p,\rho)$ to \eqref{CMHDll}, we define the energy functional to be
\begin{align}\label{energy Ell}
\Er(t):=&~\Er_4(t)+\Er_5(t)+\Er_6(t)+\Er_7(t)+\Er_8(t),\\
\Er_{4+l}(t):=&\sum_{\lee\alpha\ree=2l}\sum_{k=0}^{4-l}\left\|\eps^{(k-1)_++2l} \TT^{\alpha}\p_t^k\left( u,  B,  S,\ffpr^{\frac{(k+\alpha_0-l-3)_+}{2}}  P)\right)\right\|_{4-k-l}^2,\quad 0\leq l\leq 4.
\end{align}
\begin{rmk}
The last term in $\Er_{4+l}(t)$ looks slightly different from $E_{4+l}(t)$ for \eqref{CMHD3}, but they are essentially equivalent. In fact, when doing iteration, the variable $p$ is defined {\bf after we solve }$P$ and $B$ from \eqref{CMHDll0}, so we cannot write $p$ in the functional $\Er_{4+l}(t)$. The regularity of $p$ can be easily recovered from $p=P-\frac12|B|^2$ because $P, B, \br$ are all controlled in the above weighted anisotropic Sobolev norm.
\end{rmk}
We aim to prove the following result
\begin{prop}\label{prop Ell}
There exists some $T_1>0$ depending on $K_0$ (independent of $\eps$) such that
\[
\sup_{0\leq t\leq T_1} \Er(t)\leq C(K_0)\Er(0).
\]
\end{prop}
The proof of Proposition \ref{prop Ell} is very similar to Theorem \ref{main thm, well data}. We will skip the details for the substantially similar part and emphasize the different part. Again, we start with div-curl analysis.
\subsubsection{Div-Curl analysis}
The reduction of the spatial derivative of $P$ is the same as in Section \ref{sect reduction q}, so we omit the detail. We then start with $\Er_4(t)$. For simplicity, we only show the reduction for full spatial derivatives. We have
\begin{align}
\|u,B\|_4^2\lesssim\|u,B\|_0^2+\|\nab\times u, \nab\times B\|_3^2+\|\nab\cdot u,\nab\cdot B\|_3^2+|u_3,B_3|_{3.5}^2.
\end{align}
The boundary term vanishes thanks to the boundary condition for $u_3$ and the constraint for $B_3$. The $L^2$ estimate can be proved in the same way as in Section \ref{sect L2}. The divergence of velocity is also reduced in the same way as in Section \ref{sect div curl}. Below, we present the vorticity analysis and the treatment of $\nab\cdot B$. Do note that the magnetic field may not be divergence-free even if this is true for the initial data.

First, we analyze the divergence of the magnetic field. Taking divergence operator in the third equation of \eqref{CMHDll0}, we get
\begin{align}
\Dtr(\nab\cdot B)=(\p_i\br_j)(\p_j u_i)-(\nab\cdot \br)(\nab\cdot u)-(\p_i \ur_j)(\p_jB_i).
\end{align} We can see that the right side of the above equation only contains first-order terms without any time derivatives, so standard $H^3$ estimates give the control
\begin{align}
\ddt\|\nab\cdot B\|_3^2\lesssim P(K_0)\Er_4(t).
\end{align}

Next, we show that the hidden structure in vorticity analysis still holds for the linearized problem \eqref{CMHDll0}. The evolution equations of vorticity and current are
\begin{align}
\label{curlu ll} \rhor \Dtr(\nab\times u)-(\br\cdot\nab)(\nab\times B)=&~(\nab\rhor)\times(\Dtr u) - (\nab \br_j)\times (\p_j B)-\rhor(\nab \ur_j)\times (\p_j u), \\
\label{curlB ll}\Dtr(\nab\times B)-(\br\cdot\nab)(\nab\times u)-\br\times\nab(\nab\cdot u)=&-(\nab\times\br)(\nab\times u)-(\nab \br_j)\times(\p_j u)-(\nab \ur_j)\times(\p_j B).
\end{align}
Following the analysis in Section \ref{sect div curl}, we get 
\begin{align}
\ddt\frac12\io \rhor|\p^3\nab\times u|^2+|\p^3\nab\times B|^2\dx\leq P(K_0)\Er_4(t)+\KKr,
\end{align}where 
\begin{align}
\KKr:=\io (\p^3\nab\times B)\cdot\left(\br\times(\p^3\nab(\nab\cdot u))\right)\dx.
\end{align}
We now invoke the linearized continuity equation and the momentum equation to get (we write $\ffpr=\eps^2$ for simplicity)
\begin{equation}
\begin{aligned}
\br\times(\p^3\nab(\nab\cdot u))\lleq &~-\eps^2\br\times(\p^3\nab \Dtr P) +\eps^2\br\times((\p^3\nab \Dtr B_j)\br_j)\\
\lleq&~\eps^2\rhor \br\times(\p^3\Dtr^2 u)-\eps^2\br\times\Dtr(\br_j\p^3\p_j B)+\eps^2 \br\times\Dtr(\br_j(\p^3\nab B_j))\\
\lleq&~\eps^2\rhor \br\times(\p^3\Dtr^2 u)+\eps^2\br\times\Dtr(\br\times(\p^3\nab\times B)),
\end{aligned}
\end{equation}where we use the vector identity $(\bd{u}\times(\nab\times\bd{v}))_i=(\p_i\bd{v}_j)\bd{u}_j-\bd{u}_j\p_j\bd{v}_i$ and the $L^2$ norm of the omitted terms are directly controlled by $P(K_0)\Er_4(t)$. Thus, we get
\begin{equation}
\begin{aligned}
\KKr\lleq& \io\eps^2\rhor(\p^3\nab\times B)\cdot\left(\br\times(\p^3\Dtr^2 u)\right)\dx + \io \eps^2(\p^3\nab\times B)\cdot\left(\br\times\Dtr(\br\times(\p^3\nab\times B))\right)\dx\\
=&\io\eps^2\rhor(\p^3\nab\times B)\cdot\left(\br\times(\p^3\Dtr^2 u)\right)\dx - \io \eps^2\Dtr \left(\br\times(\p^3\nab\times B)\right)\cdot\left(\br\times(\p^3\nab\times B)\right)\dx\\
\leq &-\frac12\ddt\io\eps^2\left|\br\times(\p^3\nab\times B)\right|^2\dx + P(K_0)(\Er_4(t)+\Er_5(t)),
\end{aligned}
\end{equation}and we can conclude that
\begin{align}
\frac12\ddt\io\rhor|\p^3\nab\times u|^2+|\p^3\nab\times B|^2+\eps^2\left|\br\times(\p^3\nab\times B)\right|^2\dx\leq P(K_0)(\Er_4(t)+\Er_5(t)).
\end{align}
For $0\leq l\leq 3, 0\leq k\leq 3-l$ and the multi-index $\alpha$ satisfying $\lee\alpha\ree=2l,\alpha_3=0$, we can mimic the above steps and the analysis in Section \ref{sect div curl} to obtain that
\begin{align}
&\frac12\ddt\io\rhor|\eps^{2l+(k-1)_+}\p^{3-l-k}\nab\times\p_t^k\TT^\alpha u|^2+|\eps^{2l+(k-1)_+}\p^{3-l-k}\nab\times\p_t^k\TT^\alpha B|^2+\eps^2\left|\eps^{2l+(k-1)_+}\br\times(\p^{3-l-k}\nab\times\p_t^k\TT^\alpha B)\right|^2\dx \notag\\
\lesssim&~P(K_0)\left(\sum_{j=0}^l \Er_{4+j}(t)+\Er_{4+l+1}(t)\right).
\end{align}
Similarly, the divergence has the following bound
\begin{align}
&\left\|\eps^{2l+(k-1)_+}\nab\cdot\p_t^k\TT^\alpha u\right\|_{3-l-k}^2+\left\|\eps^{2l+(k-1)_+}\nab\cdot\p_t^k\TT^{\alpha}B\right\|_{3-l-k}^2\notag\\
\lesssim &\left\|\eps^{2l+(k-1)_++2}\p_t^k\TT^\alpha \Dtr(P, B)\right\|_{3-l-k}^2 +P(K_0)\left(\sum_{j=0}^l \Er_{4+j}(0)+\int_0^t \Er_{4+j}(\tau)\,\mathrm{d}\tau\right),
\end{align}where the first term is similarly reduced as in Section \ref{sect div curl}.

\subsubsection{Tangential estimates}
The tangential estimates are also substantially similar to the analysis in Section \ref{sect tg}. Taking $\TT^{\alpha}$ in the linearized momentum equation, we get
		\begin{equation}
			\rhor \Dtr(\TT^\alpha u)-(\br\cdot\nab)(\TT^{\alpha} B)+ \TT^{\alpha} \nab P=[\rhor \Dtr, \TT^{\alpha}]u - [\br\cdot\nab,\TT^{\alpha}]B,
		\end{equation}
		where $\TT^{\alpha}:=\TT_0^{\alpha_0}\TT_1^{\alpha_1}\TT_2^{\alpha_2}\TT_4^{\alpha_4}$ and $\lee\alpha\ree=\alpha_0+\alpha_1+\alpha_2+0\times2+\alpha_4\le 8$. As in Section \ref{sect tg}, it suffices to consider the following cases
\begin{itemize}
\item $\lee\alpha\ree=\alpha_0=4$ with $\eps^3$ weight,
\item $\lee\alpha\ree=5,~\alpha_0\geq 3$ with $\eps^4$ weight,
\item $\lee\alpha\ree=6,~\alpha_0\geq 2$ with $\eps^5$ weight,
\item $\lee\alpha\ree=7,~\alpha_0\geq 1$ with $\eps^6$ weight,
\item $\lee\alpha\ree=8$ with $\eps^8$ weight.
\end{itemize} Here we only show the sketch of $\eps^8\TP^8$ estimates and the other cases can be analyzed in the same way as in Section \ref{sect tg}. From the momentum equation, we get
		\begin{align}\label{tglls8-1}
		\frac{1}{2}\ddt\int_{\Om}\rhor|\eps^{8}\TP^8 u|^2 \dx\lleq\int_{\Om}\eps^{8}\TP^{8}(\br\cdot\nab B) \cdot(\eps^{8}\TP^{8}u ) \dx - \int_{\Om}\eps^{8}\TP^{8}(\nab P) \cdot(\eps^{8}\TP^{8}u ) \dx.
		\end{align}

The first term in \eqref{tglls8-1}, after integrating $\br\cdot\nab$ by parts and inserting the equation of $B$, becomes
		\begin{align}\label{tglls8-2}
&\int_{\Om}\eps^{8}\TP^{8}(\br\cdot\nab B) \cdot(\eps^{8}\TP^{8}u ) \dx\notag\\
\lleq&-\frac{1}{2}\ddt\int_{\Om}|\eps^{8}\TP^{8}B|^2 \dx  \underbrace{-\int_{\Om}\eps^{8} \TP^{8}B \cdot (\eps^{8} \TP^{8}(\br\nab\cdot u)) \dx }_{:=\Kr_1}.
		\end{align} The second term in \eqref{tglls8-1}, after integrating by parts and inserting the continuity equation, becomes
		\begin{align}\label{tglls8-3}
			&\int_{\Om}\eps^{8}\TP^{8} P\cdot ( \eps^{8} \TP^{8} \nab\cdot u) \dx\notag\\
		=&-\io (\eps^8\TP^8 P) (\eps^8\TP^8(\ffpr \Dtr P))\dx+\io (\eps^8\TP^8 P) (\eps^8\TP^8(\ffpr \Dtr B\cdot \br))\dx\notag\\
		\lleq&-\frac12\ddt\io|\eps^8\ffpr^{\frac12}\TP^8 P|^2\dx + \underbrace{\io \ffpr (\eps^8\TP^8 P)  \Dtr(\eps^8\TP^8 B\cdot \br)\dx}_{=:\Kr_2}
		\end{align} where the boundary intergral vanishes thanks to $u_3=B_3=0$ on $\Sigma$. We note that the commutator terms are all omitted because they can be analyzed in the same way as in Section \ref{sect tg}. Now, we shall further analyze $\Kr_1$ by inserting the continuity equation
		\begin{align}\label{tglls8-4}
		\Kr_1\lleq & \int_{\Om}\ffpr(\eps^{8} \TP^{8}B \cdot\br) (\eps^{8}\Dtr \TP^{8}P) \dx -\io \eps^8\ffp (\TP^8 B\cdot\br)\Dtr(\TP^8 B\cdot\br)\dx,
		\end{align}and thus
		\begin{align}\label{tglls8-5}
		\Kr_1+\Kr_2\lleq -\frac12\ddt\io\left|\eps^8\ffpr^{\frac12}\left(\TP^8 P-\TP^8 B\cdot\br\right)\right|^2\dx.
		\end{align}

		Combining \eqref{tglls8-1}-\eqref{tglls8-5}, we conclude that
		\begin{align}\label{tglls8-6}
			\frac12\ddt\io \rhor|\eps^{8}\TP^8 u|^2+|\eps^{8}\TP^{8}B|^2+\left|\eps^8\ffpr^{\frac12}\left(\TP^8 P-\TP^8 B\cdot\br\right)\right|^2\dx \le P(E(t)).
		\end{align}

\subsection{Picard iteration}
With the uniform-in-$(n,\eps)$ bounds for the basic state established in Proposition \ref{prop Ell}, we can now proceed the Picard iteration, that is, the strong convergence of the sequence $\{(u^\nnn, B^\nnn, S^\nnn, P^\nnn)\}_{n\in\N}$. For a function sequence $\{f^\nnn\}$, we define $[f]^\nnn:=f^\nnr-f^\nnn$. Thus, we can introduce the linear system for $\{[u]^\nnn, [B]^\nnn, [S]^\nnn, [P]^\nnn\}$ as below
\begin{equation}\label{CMHDllnn}
\begin{cases}
\rho^\nnn D_t^\nnn [u]^\nnn-(B^\nnn\cdot\nab)[B]^\nnn + \nab [P]^\nnn= -f_u^\nnn & \text{ in }[0,T]\times\Om,\\
\ffp^\nnn D_t^\nnn [P]^\nnn-\ffp^\nnn D_t^\nnn [B]^\nnn\cdot B^\nnn + \nab \cdot [u]^\nnn= -f_P^\nnn &\text{ in }[0,T]\times\Om,\\
D_t^\nnn [B]^\nnn = (B^\nnn\cdot\nab)[u]^\nnn + B^\nnn(\nab\cdot [u]^\nnn)=-f_B^\nnn  &\text{ in }[0,T]\times\Om,\\
D_t^\nnn [S]^\nnn= -f_S^\nnn  &\text{ in }[0,T]\times\Om,\\
u_3^\nnr=u_3^\nnn=B_3^\nnr=B_3^\nnn=0 &\text{ on }[0,T]\times\Sigma,\\
([u]^\nnn,[B]^\nnn,[S]^\nnn,[P]^\nnn)|_{t=0}=(\vec{0},\vec{0},0,0).
\end{cases}
\end{equation}Here $D_t^\nnn:=\p_t+u^\nnn\cdot\nab$. The source terms $f_u^\nnn,f_P^\nnn,f_B^\nnn, f_S^\nnn$ are defined as below:
\begin{align}
f_u^\nnn:=&~[\rho]^\nnl\p_t u^\nnn +[\rho u]^\nnl\cdot\nab u^\nnn - [B]^\nnl\cdot\nab B^\nnn,\\
f_P^\nnn:=&~[\ffp]^\nnl\p_tP^\nnn+[\ffp u]\cdot\nab P^\nnn \notag\\
& - ([\ffp]^\nnl\p_t B^\nnn+[\ffp u]^\nnl\cdot\nab B^\nnn)\cdot B^\nnn - \ffp^\nnl D_t^\nnl B^\nnn\cdot[B]^\nnl,\\
f_B^\nnn:=&~[u]^\nnl\cdot\nab B^\nnn - [B]^\nnl\cdot\nab u^\nnn +[B]^\nnl(\nab\cdot u^\nnn),\\
f_S^\nnn:=&~[u]^\nnl\cdot\nab S^\nnn.
\end{align}

For $n\geq 1, n\in\N^*$, we define the energy functional for the linear system \eqref{CMHDllnn} by
\begin{align}
[\Er]^\nnn(t):=&~[\Er]_3^\nnn(t)+\cdots+[\Er]_6^\nnn(t), \notag\\
[\Er]_{3+l}^\nnn(t):=&\sum_{\lee\alpha\ree=2l}\sum_{k=0}^{3-l}\left\|\eps^{(k-1)_++2l} \TT^{\alpha}\p_t^k\left( u,  B,  S,\ffpr^{\frac{(k+\alpha_0-l-2)_+}{2}} P)\right)\right\|_{3-k-l}^2,\quad 0\leq l\leq 3.
\end{align}

We aim to prove the following result
\begin{prop}\label{prop Ellnn}
There exists a time $T_2>0$ depending on $K_0$ (independent of $n,\eps$) such that
\begin{align}
\forall 2\leq n\in\N^*,\quad \sup_{0\leq t\leq T_2} [\Er]^\nnn(t)\leq \frac14\left(\sup_{0\leq t\leq T_2} [\Er]^\nnl(t)+\sup_{0\leq t\leq T_2} [\Er]^\nnll(t)\right).
\end{align}
\end{prop}
With this proposition, the local existence of the original MHD system \eqref{CMHD3} is established. Indeed, by induction on $n$, it is easy to see that
\[
\sup_{0\leq t\leq T_2} [\Er]^\nnn(t)\leq \frac{1}{2^{n-1}}P(K_0)\to 0\quad\text{ as }n\to\infty.
\]Hence, the sequence of approximate solutions $\{(u^\nnn, B^\nnn, S^\nnn, P^\nnn)\}_{n\in\N}$ has a strongly convergent subsequence whose limit is denoted by $\{(u^{(\infty)}, B^{(\infty)}, S^{(\infty)}, P^{(\infty)})\}_{n\in\N}$. Then one can define $p^{(\infty)}:=P^{(\infty)}-\frac12|B^{(\infty)}|^2$ and define $\rho^{(\infty)}$ from the equation of state $p^{(\infty)}=p(\rho^{(\infty)}, S^{(\infty)})$. These limit functions exactly satisfy the original MHD system \eqref{CMHD3}, which leads to the local existence and uniform-in-$\eps$ estimates of \eqref{CMHD3}. The uniqueness follows from a parallel argument as in Proposition \ref{prop Ellnn}.

\begin{proof}[Sketch of the proof of Proposition \ref{prop Ellnn}]
The proof is again substantially the same as Proposition \ref{prop Ell}, so we no longer repeat the tedious details. In particular, the energy $[\Er]^\nnn(t)$ gives a weighted anisotropic Sobolev norm in $H_*^6(\Om)$. Also, we note that the source terms $f_u, f_P, f_B, f_S$ only contain the derivatives of the basic state up to 1st order. Therefore, these source terms do not bring any loss of regularity. Here, we only show that the hidden structure in the vorticity analysis still holds and omit all the other steps. The vorticity-current equations now reads
\begin{align}
&\rho^\nnn D_t^\nnn(\nab\times [u]^\nnn) - (B^\nnn\cdot\nab)(\nab\times [B]^\nnn) \notag \\
\label{curlu llnn}= &-\nab\times f_u^\nnn -\nab\rho^\nnn\times D_t^\nnn [u]^\nnn + (\nab B_j^\nnn)\times(\p_j [B]^\nnn)+ \rho^\nnn[D_t^\nnn,\nab\times][u]^\nnn,\\
&D_t^\nnn(\nab\times[B]^\nnn)-(B^\nnn\cdot\nab)(\nab\times[u]^\nnn)-\underline{B^\nnn\times(\nab(\nab\cdot[u]^\nnn))}\notag\\
\label{curlB llnn}= &-\nab\times f_B^\nnn + [D_t^\nnn,\nab\times][B]^\nnn + (\nab B_j^\nnn)\times(\p_j [u]^\nnn) - (\nab\times B^\nnn)(\nab\cdot[u]^\nnn).
\end{align}
We notice that the underlined term has the same structure as in \eqref{curlB ll}, so we can again still mimic the analysis for the linearized problem \eqref{CMHDll0} to treat this term in the vorticity analysis.
\end{proof}

\end{appendix}

	\end{document}